\documentclass[11pt, reqno]{amsart}
\usepackage{eucal}

\usepackage{graphicx} 
\usepackage{comment}
\usepackage{color}

\newcommand{\R}{\mathbb{R}}
\newcommand{\C}{\mathbb{C}}
\newcommand{\id}{\operatorname{Id}}
\newcommand{\msf}[1]{\mathsf{#1}}

\newcommand{\mc}[1]{\mathcal{#1}}
\newcommand{\cop}{\mc C}

\usepackage[colorlinks, linkcolor=black, citecolor=blue, linktocpage,urlcolor=blue]{hyperref}
\theoremstyle{definition}
\newtheorem{thm}{Theorem}[section]
\newtheorem{definition}[thm]{Definition}
\newtheorem{proposition}[thm]{Proposition}

\newtheorem{lemma}[thm]{Lemma}

\newtheorem{remark}[thm]{Remark}

\author{Kyle Broder}
\address{The University of Queensland,  St. Lucia,  QLD 4067, Australia}
\email{k.broder@uq.edu.au}
\author{Jan Nienhaus}
\address{University of California, Los Angeles, 520 Portola Plaza, CA, 90095}
\email{nienhaus@math.ucla.edu}
\author{Peter Petersen}
\address{University of California, Los Angeles, 520 Portola Plaza, CA, 90095}
\email{petersen@math.ucla.edu}
\author{James Stanfield}
 \address{Cluster of Excellence Mathematics M\"unster, Orléans-Ring 10, 48149 Münster, Germany}
 \email{james.stanfield@uni-muenster.de} 
\author{Matthias Wink}
\address{University of California, Santa Barbara, South Hall, CA, 93106}
\email{wink@math.uscb.edu}

\usepackage{fancyhdr}

\keywords{K\"ahler manifolds, Bochner technique, Vanishing theorems}

\makeatletter
\@namedef{subjclassname@2020}{
\textup{2020} Mathematics Subject Classification}
\makeatother
\subjclass[2020]{32Q10, 32Q15, 32Q20, 53C21, 53C55}

\begin{document}
\newcommand{\Ext}{\bigwedge\nolimits}
\newcommand{\Div}{\operatorname{div}}
\newcommand{\Hol} {\operatorname{Hol}}
\newcommand{\diam} {\operatorname{diam}}
\newcommand{\Scal} {\operatorname{Scal}}
\newcommand{\scal} {\operatorname{scal}}
\newcommand{\Ric} {\operatorname{Ric}}
\newcommand{\Hess} {\operatorname{Hess}}
\newcommand{\grad} {\operatorname{grad}}
\newcommand{\Rm} {\operatorname{Rm}}
\newcommand{ \Rmzero } {\mathring{\Rm}}
\newcommand{\Rc} {\operatorname{Rc}}
\newcommand{\Curv} {S_{B}^{2}\left( \mathfrak{so}(n) \right) }
\newcommand{ \tr } {\operatorname{tr}}
\newcommand{ \Riczero } {\mathring{\Ric}}
\newcommand{ \Ad } {\operatorname{Ad}}
\newcommand{ \dist } {\operatorname{dist}}
\newcommand{ \rank } {\operatorname{rank}}
\newcommand{\Vol}{\operatorname{Vol}}
\newcommand{\dVol}{\operatorname{dVol}}
\newcommand{ \zitieren }[1]{ \hspace{-3mm} \cite{#1}}
\newcommand{ \pr }{\operatorname{pr}}
\newcommand{\diag}{\operatorname{diag}}
\newcommand{\Lagr}{\mathcal{L}}
\newcommand{\av}{\operatorname{av}}
\newcommand{ \floor }[1]{ \lfloor #1 \rfloor }
\newcommand{ \ceil }[1]{ \lceil #1 \rceil }
\newcommand{\Sym} {\operatorname{Sym}}
\newcommand{\bcirc}{ \ \bar{\circ} \ }
\newcommand{\sign}[1]{\operatorname{sign}(#1)}
\newcommand{\cone}{\operatorname{cone}}
\newcommand{\pbd}{\varphi_{bar}^{\delta}}
\newcommand{\End}{\operatorname{End}}

\renewcommand{\labelenumi}{(\alph{enumi})}
\newtheorem{maintheorem}{Theorem}[]
\renewcommand*{\themaintheorem}{\Alph{maintheorem}}
\newtheorem*{remark*}{Remark}

\vspace*{-1cm}

\title{Vanishing theorems for Hodge numbers \\
and the Calabi curvature operator}

\begin{abstract}
It is shown that a compact $n$-dimensional K\"ahler manifold with $\frac{n}{2}$-positive Calabi curvature operator has the rational cohomology of complex projective space. For even $n,$ this is sharp in the sense that the complex quadric with its symmetric metric has $\frac{n}{2}$-nonnegative Calabi curvature operator, yet $b_n =2.$ Furthermore, the compact K\"ahler manifolds with an $\frac{n}{2}$-nonnegative Calabi curvature operator are classified. 

In addition, the previously known results for the K\"ahler curvature operator are improved when the metric is K\"ahler--Einstein.
\end{abstract}

\maketitle

\pagestyle{fancy}
\fancyhead{} 
\fancyhead[CE]{Vanishing theorems for Hodge numbers and the Calabi curvature operator}
\fancyhead[CO]{Kyle Broder, Jan Nienhaus, Peter Petersen, James Stanfield, Matthias Wink}
\fancyfoot{}
\fancyfoot[CE,CO]{\thepage}

\tableofcontents

\section{Introduction}
\noindent The Bochner technique is an analytic method for understanding how curvature interacts with the topology, and in particular the cohomology, of a Riemannian manifold. Its role in proving vanishing theorems in K\"ahler geometry goes back to Bochner \cite{Bochner}, who showed that a compact K\"ahler manifold with $k$-positive Ricci curvature does not admit holomorphic $p$-forms for $k \leq p \leq n$. This has inspired several results in this direction, see for example \cite{BishopGoldberg,Bouche, Demailly,GreeneWu,KobayashiWu, NiZhengComparisonAndVanishing, NiZhengPositivityAndKodaira, Yang}.

Recall that a self-adjoint operator with eigenvalues $\lambda_1 \leq \ldots \leq \lambda_m$ is said to be \textit{$k$-nonnegative} for some $k \geq 1$ if $$\lambda_1 + \ldots + \lambda_{\lfloor k \rfloor} + (k - \lfloor k \rfloor) \lambda_{\lfloor k \rfloor +1} \ \geq \ 0,$$ and \textit{$k$-positive} if the above inequality is strict. 

In \cite{PW2}, the third and fifth named authors obtained a cohomological characterization of $\mathbf{P}^n$ in terms of eigenvalues of the K\"ahler curvature operator $\mathfrak{K} \colon \Omega^{1,1}_X \to \Omega^{1,1}_X.$ Namely, if $X^n$ is a compact K\"ahler manifold with $\left(3-\frac{2}{n}\right)$-positive K\"ahler curvature operator, then $X$ has the rational cohomology of $\mathbf{P}^n.$ The results in \cite{PW2} in fact establish vanishing results for individual Hodge numbers $h^{p,q}(X) := \dim_{\mathbf{C}} H^q(X, \Omega_X^p).$ In particular, for $(p,p)$-forms one only requires $(n+1-p)$-positive K\"ahler curvature operator to conclude $h^{p,p}(X)=1$ while the curvature conditions for $(p,0)$-forms are stronger than those required by  Bochner.

In the present article, we overcome this obstacle by considering the \emph{Calabi curvature operator} $\mathcal{C} \colon \mathcal{S}_X^{2,0} \to \mathcal{S}_X^{2,0}$, defined by 
\begin{align*}
g(\mathcal{C}(\xi \odot \mu), \bar{\eta} \odot \bar{\nu}) := 4 R(\xi, \bar{\eta}, \bar{\nu}, \mu), 
\end{align*}
where $\xi \odot \mu : = \xi \otimes \mu + \mu \otimes \xi \in \mathcal{S}^{2,0}_X := T^{1,0}X \odot T^{1,0}X$.

The Calabi curvature operator first appears in a paper of Calabi--Vesentini \cite{CV} and subsequently appeared in \cite{BCV,OT, Sitaramayya}. Our first main result is the following.

\begin{maintheorem}
\label{thm:main-theorem-1}
Every compact $n$-dimensional K\"ahler manifold with $\frac{n}{2}$-positive Calabi curvature operator has the rational cohomology of $\mathbf{P}^n$.
\end{maintheorem}

For $n$ even, this result is sharp, as the complex quadric $\msf{SO}(n+2)/ \left(\msf{SO}(2) \times \msf{SO}(n) \right)$ has $\frac{n}{2}$-nonnegative Calabi curvature operator and Betti number $b_n=2.$ Moreover, for $n$ odd, the complex quadric is a rational cohomology $\mathbf{P}^n.$ 

Theorem~\ref{thm:main-theorem-1} was previously known for positive, i.e., $1$-positive, Calabi curvature operator \cite{OT}. In fact, since positive Calabi curvature operator implies positive holomorphic bisectional curvature, see Remark~\ref{rmk:NonNegativeCalabi}, K\"ahler manifolds with $\cop > 0$ are biholomorphic to $\mathbf{P}^n$ \cite{Mori,SiuYau}. This raises the question of whether a compact K\"ahler manifold with $\frac{n}{2}$-positive Calabi curvature operator is biholomorphic to $\mathbf{P}^n$.

The rigidity case of Theorem~\ref{thm:main-theorem-1} is established in the following result.

\begin{maintheorem}
\label{thm:main-theorem-2}
If $X^n$ is a compact K\"ahler manifold with $\frac{n}{2}$-nonnegative Calabi curvature operator, then one of the following holds: \begin{itemize}
\item[(i)] The holonomy is irreducible and $X$ has either the rational cohomology of $\mathbf{P}^n$ or is isometric to the complex quadric $\msf{SO}(n+2)/ \left(\msf{SO}(2) \times \msf{SO}(n) \right)$ with its symmetric metric.
\item[(ii)] The holonomy is reducible, and a finite cover is isometric to $\mathbf{T}^k \times Y$,  where $Y$ is a product of spaces that are biholomorphic to projective spaces.
\end{itemize}
\end{maintheorem}

Theorem~\ref{thm:main-theorem-2} generalizes to a local holonomy classification for (not necessarily complete) K\"ahler manifolds in the spirit of \cite{NPWW}, see Theorem~\ref{thm:main-theorem-2.2}.


Theorem~\ref{thm:main-theorem-1} follows after establishing vanishing of all the primitive harmonic forms. Note that by the Lefschetz and Hodge decomposition theorems, this gives complete control of the Hodge and Betti numbers. Since a $(p,q)$-form can be decomposed into real and imaginary parts, it suffices to prove that \textit{real} primitive harmonic forms in $\Omega^{p,q}_X \oplus \Omega^{q,p}_X$ vanish. This turns out to be key in proving the required estimates, see Proposition~\ref{MainEstimateForForms}. To state our refined vanishing theorem, we define
\begin{align*}
\Upsilon_{p,q}  \ : = \ \frac{(p+q)(n+1) - 2pq}{2+4\min \left( p, q, \frac{\sqrt{pq}}{2} \right)},
\end{align*} 
for all $1 \leq p,q\leq n$. Our main theorem for individual Hodge numbers is the following result.

\begin{maintheorem}\label{thm:main-theorem-3}
Let $X^n$ be a compact K\"ahler manifold. If the Calabi curvature operator $\mathcal{C}$ is $\Upsilon_{p,q}$-nonnegative, then the real primitive harmonic forms in $\Omega_X^{p,q} \oplus \Omega_X^{q,p}$ are parallel. If $\mathcal{C}$ is $\Upsilon_{p,q}$-positive, then the real primitive harmonic forms in $\Omega_X^{p,q} \oplus \Omega_X^{q,p}$ vanish. 
\end{maintheorem}

Note that $\Upsilon_{p,q} \geq \frac{n}{2}$ whenever $1 \leq p + q \leq n$, hence Theorem~\ref{thm:main-theorem-1} follows from the above discussion and Serre duality. We also note that $\Upsilon_{p,q}$ is minimized when $p=q=1$. Further, $\Upsilon_{p,p} = \frac{p((n+1)-p)}{1+p}$ and $\Upsilon_{p,0} = \frac{p(n+1)}{2}.$ In particular, $\Upsilon_{n,0} = \dim \mathcal{S}_X^{2,0}$, which recovers the result of Bochner \cite{Bochner} that $h^{n,0}$ is controlled by the scalar curvature; this contrasts the previous work \cite{PW2}. 
\vspace{2mm}

It is worth noting that topological invariants such as characteristic classes only contribute to Hodge numbers that occur in degree $(p,p)$. Moreover, K\"ahler manifolds with symmetries, such as flag varieties \cite{CG} and, in particular, compact Hermitian symmetric spaces \cite{BH} only have nonzero Hodge numbers in degree $(p,p)$. The same is true for compact K\"ahler manifolds admitting a Killing field with discrete vanishing locus \cite{WrightKillingVectorFieldsAndHarmonicForms}. Other interesting examples of this type, such as twistor spaces, can be found in \cite{CLS,LS,Salamon}.

\begin{remark*}
    \normalfont
    The Ricci tensor is bounded from below by the sum of the lowest $n$ eigenvalues of the Calabi curvature operator. In particular, if the sum of the lowest $\Upsilon_{p,q}$ eigenvalues of the Calabi curvature operator is bounded from below, we automatically obtain a lower bound on the Ricci curvature. This allows us to obtain estimates for Hodge numbers of compact K\"ahler manifolds in terms of their upper diameter bound, as in \cite[Theorem D]{PW2}. 
\end{remark*}

With regard to the setup and the use of the Calabi curvature operator, we would like to note that it was already observed by Calabi--Vesentini \cite{CV} that the curvature tensor induced by an algebraic Calabi curvature operator, i.e., a self-adjoint operator on $\mathcal{S}_p^{2,0},$ automatically satisfies the algebraic Bianchi identity. Therefore, the setup of the Calabi curvature operator allows us to incorporate the algebraic Bianchi identity, which is a significant improvement over the setup in \cite{PW2}. The price to pay is that elements of $\mathcal{S}_X^{2,0}$ acting as endomorphisms do not preserve the type of the form. More precisely, in contrast with the $\mathfrak{u}(n)$-action on $(p,q)$-forms considered in \cite{PW2}, the operators we consider transform $(p,q)$-forms into $(p-1,q+1)$-forms. 

Furthermore, contrasting \cite{NPWBettiNumbersAndCOSK}, the Calabi curvature operator is different from the curvature operator of the second kind, which acts on trace-free symmetric tensors, as the Calabi curvature operator is defined only on (smooth) sections of the holomorphic bundle $\mathcal{S}_X^{2,0}$ and not all symmetric complex tensors. 

A technical outline of the proof is given at the beginning of Section~\ref{SectionProofMainTheorems}. \vspace{2mm}

In addition to the results for the Calabi curvature operator, we prove an analog of Theorem~\ref{thm:main-theorem-1} for the K\"ahler curvature operator $\mathfrak{K} \colon \Omega_X^{1,1} \to \Omega_X^{1,1}$ of a K\"ahler--Einstein metric. As explained at the beginning of the introduction, the results in \cite{PW2} for general K\"ahler manifolds required $\left(3 - \frac{2}{n} \right)$-positive K\"ahler curvature operator to obtain vanishing of all primitive forms, while all primitive $(p,p)$-forms already vanish if $\mathfrak{K}$ is $\left( \frac{n}{2}+1 \right)$-positive. 

In the case of K\"ahler--Einstein manifolds, we note that the K\"ahler form is an eigenvector with the Einstein constant as the corresponding eigenvalue. Therefore, we may restrict the K\"ahler curvature operator to primitive $(1,1)$-forms and obtain the following result.

\begin{maintheorem}
\label{thm:main-theorem-4}
Let $(X^n,g)$ be a compact K\"ahler--Einstein manifold. If the restriction of the K\"ahler curvature operator to primitive $(1,1)$-forms $\mathfrak{K}|_{\mathfrak{su}(n)} \colon \Ext^{1,1}_0 TX \to \Ext^{1,1}_0 TX$ is $\left(\frac{n}{2} +1 \right)$-nonnegative, then all primitive harmonic forms are parallel. If $\mathfrak{K}|_{\mathfrak{su}(n)}$ is $\left(\frac{n}{2} +1 \right)$-positive, then all primitive harmonic $(p,q)$-forms vanish.
\end{maintheorem}

In Section~\ref{SectionKaehlerEinstein}, we in fact establish curvature conditions depending on $(p,q)$ so that the corresponding Hodge number vanishes, similar to Theorem~\ref{thm:main-theorem-3}. The proof relies on extending the techniques in \cite{PW2} for K\"ahler--Einstein metrics.

\subsection*{Acknowledgments}
K.~Broder and P.~Petersen thank Ramiro Lafuente and the University of Queensland (UQ) for funding P.~Petersen's research visit to UQ and for their hospitality.
K.~Broder, P.~Petersen, and J.~Stanfield would also like to thank MATRIX for their hospitality during the workshop on Geometric Analysis and Symmetry in February 2025. 

K.~Broder was supported by the Australian Research Council Discovery Projects DP220102530 and DP240101772.  J. Stanfield was supported by the Deutsche Forschungsgemeinschaft (DFG, German Research Foundation) under Germany's Excellence Strategy EXC 2044 –390685587, Mathematics Münster: Dynamics–Geometry–Structure and by the CRC 1442 “Geometry: Deformations and Rigidity” of the DFG.

\section{Preliminaries and conventions}

\subsection{Conventions on tensors and norms} 
Let $(V,g)$ be a Euclidean $\mathbf{R}$-vector space. We define the wedge product as in \cite{PetersenManifoldTheory}. In particular, for $v_1, \ldots, v_k \in V$ we have 
\begin{align*}
    v_1 \wedge \ldots \wedge v_k \ :=  \ \sum_{\sigma \in S_k} \sign{\sigma} \ v_{\sigma(1)} \otimes \ldots \otimes v_{\sigma(k)}
\end{align*}
and specifically for $v, w \in V$, $v \wedge w \ : = \ v \otimes w - w \otimes v$. Similarly, for $v,w \in V$, we set $v \odot w := v \otimes w + w \otimes v$. Using the metric $g$,  the map $v \otimes w \mapsto g(v, \cdot )w$ identifies  $V \otimes V \simeq V^{*} \otimes V$.  The canonical isomorphism $V^{*} \otimes V \simeq \End(V)$ then yields 
\begin{eqnarray*}
    (v \otimes w) z & = & g(v,z) w, \\
    (v \wedge w) z & = & g(v,z)w - g(w,z)v, \\
    (v \odot w)z & = & g(v,z)w + g(w,z)v.
\end{eqnarray*} This induces an identification of $\Ext^2V$ with $\mathfrak{so}(V)$ and of $\bigodot^2V$ with the vector space of self-adjoint endomorphisms of $V$.

For a tensor $T \in \bigotimes^k V \otimes \bigotimes^l V^{*}$ we use the tensor norm, 
\begin{align*}
    |T|^2 
    = \sum_{\substack{ {i_1},  \ldots, {i_k} \\ {j_1}, \ldots, {j_l}}} \left( T^{{i_1}  \ldots {i_k}} {}_{{j_1} \ldots {j_l}} \right)^2.
\end{align*}
Consequently, if $e_1, \ldots, e_k$ are orthonormal, then 
\begin{align*}
 |e_1 \wedge e_2 |^2 = | e_1 \odot e_2 |^2  = 2 \text{ and } | e_1  \wedge \ldots \wedge e_k |^2   =  k!.
\end{align*}

Declaring the identification  $V \otimes V \simeq V^{*} \otimes V \simeq \End(V)$ to be an isometry, the induced inner product on $\End(V)$ is 
\begin{align*}
    g(L_1, L_2) \ = \ \tr(L_1 L_2^*).
\end{align*}

\begin{lemma}
    (a) If $L \in \End(V) \simeq V \otimes V$ and $u,v \in V,$ then $g(L, u \otimes v) = g(Lu, v).$

    (b) If $L \in \mathfrak{so}(V)$ and $u,v \in V,$ then $g(L, u \wedge v) = 2 g(Lu,v).$
\end{lemma}
\begin{proof} 
Note that $g(z \otimes w, u \otimes v) = g(z,u)g(w,v) = g( (z \otimes w)u, v).$ This implies (a). For (b), if $L \in \mathfrak{so}(V),$ then $g(L, u \wedge v) = g(L, u \otimes v - v \otimes u) = g(Lu,v) - g(Lv, u) = 2 g(Lu,v).$
\end{proof}

\subsection{Complexification} Let $J$ be a compatible (almost) complex structure on $(V,g),$ i.e., $J^2 = - \id$ and $g(J \cdot, J \cdot ) = g( \cdot, \cdot ).$ Let $\dim_{\mathbf{R}} V = 2n$ and let $e_1, \ldots, e_{2n}$ be a $J$-adapted orthonormal basis for $V.$ By convention, $Je_a = e_{a+n}$ for $a=1, \ldots, n.$ Consider $V^{\mathbf{C}} = V \otimes_{\mathbf{R}} \mathbf{C}$ and extend all $\mathbf{R}$-linear maps to be $\mathbf{C}$-linear. Note that 
\begin{align*}
    \tr_{\mathbf{R}}(L) = \tr_{\mathbf{C}}(L^{\mathbf{C}})=2 \tr_{\mathbf{R}}(L^{\mathbf{C}}),
\end{align*}
where $L^{\mathbf{C}}$ denotes the $\mathbf{C}$-linear extension. In the following, we will not differentiate between $L$ and $L^{\mathbf{C}}.$ With this convention, the metric $g$ is extended $\mathbf{C}$-bilinearly and we obtain the identification
\begin{align*}
    \End( V^{\mathbf{C}} ) \simeq  \left( V^{\mathbf{C}} \right)^{*} \otimes V^{\mathbf{C}} \simeq V^{\mathbf{C}} \otimes V^{\mathbf{C}}.
\end{align*}

The space $V^{\mathbf{C}}$ decomposes into the $\pm \sqrt{-1}$-eigenspaces of $J$, i.e., $V^{\mathbf{C}} = V^{1,0} \oplus V^{0,1}$. The corresponding $(1,0)$- and $(0,1)$-vectors are denoted by
\begin{align*}
    Z_a \  := \ \frac{1}{\sqrt{2}} \left( e_a - \sqrt{-1} J e_a \right), \hspace{0.5cm} \overline{Z_a} \ : = \  \frac{1}{\sqrt{2}} \left( e_a + \sqrt{-1} J e_a \right) 
\end{align*}
for $a = 1, \ldots, n$ and we obtain a basis $Z_1, \ldots, Z_n$ for $V^{1,0}$ and a basis $\overline{Z_1}, \ldots, \overline{Z_n}$ for $V^{0,1}=\overline{V^{1,0}}.$ Note that \begin{align*}
    g(Z_a, Z_b) \ = \ g(\overline{Z_a}, \overline{Z_b}) \ = \ 0,  \hspace{0.5cm} g(Z_a, \overline{Z_b}) \ = \ \delta_{ab}.
\end{align*}
Furthermore, we obtain the decomposition $\Ext^2 V^{\mathbf{C}} = \Ext^{2,0} V \oplus \Ext^{1,1} V \oplus \Ext^{0,2} V$ from the basis
\begin{align*}
    \frac{1}{\sqrt{2}} Z_a \wedge Z_b & \text{ for } 1 \leq a < b \leq n, \\
    \frac{1}{\sqrt{2}} Z_a \wedge \overline{Z_b} & \text{ for } 1 \leq a , b \leq n, \\
    \frac{1}{\sqrt{2}} \overline{Z_a} \wedge \overline{Z_b} & \text{ for } 1 \leq a < b \leq n.
\end{align*}

\begin{remark} \normalfont
We also raise indices using the metric according to the formulae $Z^a = g(\cdot, \overline{Z_a})$ and $ \overline{Z^a} = g(Z_a, \cdot )$. In particular, note that for $L \in \End(V^{\C}), \bar{L}(v)= \overline{L(\bar{v})}$ is complex linear.
\end{remark}

\begin{remark}
    \label{TraceInZCoordinates}
    \normalfont
    Note that if $A \in \bigotimes^2 V^{*},$ then
    \begin{align*}
        \sum_{j=1}^{2n} A(e_j, e_j)  = \sum_{a=1}^n \left( A(Z_a, \overline{Z_a}) + A(\overline{Z_a}, Z_a) \right).
    \end{align*}
\end{remark}

\subsection{Curvature tensors}
\label{SectionCurvatureTensors}
An algebraic curvature tensor is an element $R \in \bigotimes^4 V^{*}$ such that 
\begin{align*}
    R(X,Y,Z,W)=-R(Y,X,Z,W)=-R(X,Y,W,Z)=R(Z,W,X,Y)
\end{align*}
which satisfies the algebraic Bianchi identity. 

Furthermore, we define $R(X,Y)Z$ via $g(R(X,Y)Z,W)=R(X,Y,Z,W)$ and the {\em curvature operator} via
\begin{align*}
    & \mathfrak{R} \colon \Ext^2V \to \Ext^2 V, \\
    g( \mathfrak{R}(X \wedge Y) &,  Z \wedge W) \ : = \ 2 R(X,Y,Z,W).
\end{align*}
With this convention, the curvature operator of the round sphere is the identity. \vspace{2mm}

Let $(V,J,g)$ be a Euclidean vector space with a compatible complex structure. An algebraic curvature tensor is an {\em algebraic K\"ahler curvature tensor} if 
\begin{align*}
    R(X,Y,Z,W) \ = \ R(JX,JY,Z,W) \ = \ R(X,Y,JZ,JW). 
\end{align*}
It follows that $R(Z_a,Z_b) = R(\overline{Z_a},\overline{Z_b}) = 0$ and $R(Z_a,\overline{Z_b}) \in \End(V^{\mathbf{C}})$ induces endomorphisms of $V^{1,0}$ and $V^{0,1},$ respectively. Since the endomorphism $R( \cdot ,\cdot ) \in \End(V^{\mathbf{C}})$ is defined via $g(R(X,Y)Z,W)=R(X,Y,Z,W)$, we see that $R(X,Y)$ is skew-adjoint with respect to $g$. In particular, with the identification $\mathfrak{so}(V^{\mathbf{C}})=\Ext^2 V^{\mathbf{C}},$ we have
\begin{align*}
    R(Z_a, \overline{Z_b}) = \frac{1}{2} \sum_{c,d} g \left(R(Z_a,\overline{Z_b}), \overline{Z_c \wedge \overline{Z_d}} \right) Z_c \wedge \overline{Z_d}
\end{align*}
for a K\"ahler curvature tensor $R.$ Furthermore, we note that a K\"ahler curvature tensor has the symmetries
\begin{align*}
    R(X, \overline{Y}, Z, \overline{W}) =R(Z, \overline{Y}, X, \overline{W}) = R(X, \overline{W}, Z, \overline{Y}).
\end{align*}
for $X,Y,Z,W \in V^{1,0}.$

Recall that the curvature operator of a Riemannian manifold vanishes on the orthogonal complement of the holonomy algebra. For K\"ahler manifolds, as $\mathfrak{u}_{\mathbf{C}} (n) \simeq \Ext^{1,1}V$, this induces the K\"ahler curvature operator. 

\begin{definition}
    For an algebraic K\"ahler curvature operator $R$ we define the induced {\em K\"ahler curvature operator} $\mathfrak{K} \colon \Ext^{1,1} V \to \Ext^{1,1} V$ by
\begin{align*}
    g( \mathfrak{K}(X \wedge \overline{Y}) &,  Z \wedge \overline{W}) \ : =  \ 2 R(X,\overline{Y},Z,\overline{W})
\end{align*}
for all $X,Y,Z,W \in V^{1,0}.$
\end{definition}

In addition, an algebraic K\"ahler curvature tensor $R$ also induces a self-adjoint operator on  $\bigodot^2 V^{1,0}.$ It was introduced and studied by Calabi--Vesentini in \cite{CV}. In honor of Calabi, we call it the {\em Calabi curvature operator}.

\begin{definition}
    For an algebraic K\"ahler curvature operator $R$ we define the induced {\em Calabi curvature operator} $\mathcal{C} \colon \bigodot^2 V^{1,0} \to \bigodot^2 V^{1,0}$ by 
    \begin{eqnarray*}
        g( \mathcal{C}( X \odot Y)  , \overline{Z} \odot \overline{W} ) &:=& 4 R(X, \overline{Z},  \overline{W}, Y)
    \end{eqnarray*}
for $X,Y,Z,W \in V^{1,0}.$
\end{definition}

\begin{remark}
Both the K\"ahler curvature operator $\mathfrak{K}$ and the Calabi curvature operator $\mathcal{C}$ are self-adjoint with respect to $g.$

If $(X,g)$ is a K\"ahler manifold with constant holomorphic sectional curvature, then the Calabi curvature operator of $g$ is a scalar multiple of the identity. Specifically, the Calabi curvature operator of $\mathbf{P}^n$ with the Fubini--Study metric and holomorphic sectional curvature one
is the identity operator. 
\end{remark}

\begin{remark} \label{rmk:NonNegativeCalabi}
Nonnegativity (respectively, positivity) of the Calabi curvature operator implies that the holomorphic bisectional curvature is nonnegative (respectively, positive). Hence, compact K\"ahler manifolds with a positive Calabi curvature operator are classified by the Mori \cite{Mori} and Siu--Yau \cite{SiuYau} resolution of the Frankel--Hartshorne conjecture. Compact Hermitian symmetric spaces of rank $\geq 2$ do not have nonnegative Calabi curvature operator \cite{CV,BCV}. Hence, Mok's theorem \cite{Mok} implies that the universal cover of a compact Kähler manifold with nonnegative Calabi curvature operator is the product of $\mathbf{C}^k$ and complex projective spaces. 
\end{remark}

\begin{remark}
It was observed by Calabi--Vesentini \cite{CV} (see also \cite{Sitaramayya}) that the space of algebraic K\"ahler curvature operators on $V$ is in one-to-one correspondence with the space of self-adjoint operators of $\bigodot^2 V^{1,0}$. On the other hand, a given self-adjoint map of $\Ext^{1,1}V$ does not always correspond to a K\"ahler curvature operator, as the algebraic Bianchi identity may fail.
\end{remark}

\begin{remark}
\normalfont
If $\Sigma_{\nu} \in \bigodot^2 V^{1,0}$ is an eigenbasis for the Calabi curvature operator $\mathcal{C}$ with corresponding eigenvalues $\sigma_{\nu}$ such that $g(\Sigma_{\nu}, \overline{\Sigma_{\mu}})=\delta_{\nu \mu},$ then
\begin{align*}
    R(Z_a, \overline{Z_b}) = - \sum_{\nu} \sigma_{\nu} \overline{\Sigma_{\nu}} Z_a \wedge \Sigma_{\nu} \overline{Z_b}.
\end{align*}

\end{remark}

\subsection{The curvature term of the Lichnerowicz Laplacian}
Let $\varphi \in \Omega_{\mathbf{C}}^k X$ be a smooth $k$-form on a compact K\"ahler manifold $X$.  Here,  we take $\varphi$ to be complex-valued,  but the Bochner technique also finds important applications to bundle-valued $k$-forms (see, e.g.,  \cite{BroderStanfield, GreeneWu,KobayashiWu,YangHSC}).  If $\Delta_{g}$ denotes the Hodge Laplacian of a K\"ahler metric $g,$ then the Bochner--Weitzenb\"ock formula asserts that \begin{eqnarray*}
\Delta_{g} \varphi &=& (\text{d} \text{d}^{\ast} + \text{d}^{\ast} \text{d}) \varphi \ = \ \nabla^{\ast} \nabla \varphi + \text{Ric}_L(\varphi),
\end{eqnarray*} where $\text{Ric}_L(\varphi)$ is tensorial and a contraction of $\nabla^2 \varphi$,  and therefore,  must be a contraction of $R \otimes \varphi$. Explicitly, it is given by the formula \begin{eqnarray*}
\text{Ric}_L(\varphi)(\xi_1, ..., \xi_k) & : = &   \sum_{s=1}^k \sum_{j=1}^{2n} \left( R(\xi_s, e_j) \varphi  \right) (\xi_1, ..., \xi_{s-1} ,e_j, \xi_{s+1}, ..., \xi_k)
\end{eqnarray*} 
where $e_1, \ldots, e_{2n}$ is an orthonormal basis for $T_pX.$ Here, $R( \cdot, \cdot) \in \End(TX_{\mathbf{C}})$ is considered as a skew-adjoint endomorphism acting as a derivation. Specifically, for a linear map $L$ and a tensor $T$ we have
\begin{align}
\label{DefinitionActionOnTensors}
    (LT)(\xi_1, \ldots, \xi_k) = - \sum_{i=1}^k T(\xi_1, \ldots, L\xi_i, \ldots, \xi_k).
\end{align}

We emphasize that the subscript $L$ indicates that $\text{Ric}_L$ is the curvature term that appears in the \textit{Lichnerowicz Laplacian} $\Delta_L : = \nabla^{\ast} \nabla + c \text{Ric}_L$,  for some $c>0$ (see, e.g., \cite{PetersenRiemannianGeometry}).  If,  in addition,  $\varphi$ is harmonic,  then 
\begin{eqnarray}
\label{prelim:Bochner-formula}
0 &=& g \left( \nabla^{\ast} \nabla \varphi, \overline{\varphi} \right) + g\left( \text{Ric}_L(\varphi), \overline{\varphi} \right).
\end{eqnarray} 
If $g \left( \text{Ric}_L(\varphi), \overline{\varphi} \right) \geq 0$, the divergence theorem or maximum principle can be used to show that $\varphi$ is parallel, and thus to estimate Betti numbers.

The precise form of the curvature term appearing in \eqref{prelim:Bochner-formula} depends on which irreducible tensor module $\varphi$ belongs to. If $\varphi \in \Omega_X^{1,0}$, then $g \left( \operatorname{Ric}_L(\varphi), \overline{\varphi} \right) = \Ric_g( \varphi, \overline{\varphi}),$
while if $\varphi \in \Omega_X^{n,0}$,  $g \left( \text{Ric}_L(\varphi), \overline{\varphi} \right) = \scal_{g} | \varphi |^2$. In general, this curvature term does not afford such a kind expression. For $(1,1)$-forms, it corresponds to quadratic orthogonal bisectional curvature \cite{BishopGoldberg,BochnerYano, WuYauZheng}.  

If the curvature term is positive on an irreducible module, then this forces $\varphi \equiv 0$, and thus one obtains vanishing theorems. Note that the curvature term may have a natural kernel; for example, in the K\"ahler case, the powers of the K\"ahler form.

\section{Bochner formulae and a calculus for curvature tensors}

In this section, we present a new technique to derive the curvature term for the Lichnerowicz Laplacian on forms using the Calabi curvature operator. The formula established in Lemma~\ref{CurvatureTermWithCalabi} first appeared (without proof) in the work of Ogiue--Tachibana \cite{OT}. Our new approach is based more generally on curvature operators on tensor bundles and also recovers the corresponding identity for the curvature operator of the second kind in \cite{NPWBettiNumbersAndCOSK}.

\begin{definition} Let $R$ be an algebraic curvature tensor on $(V,g).$  For $i=1,2$ the curvature operators $\mathcal{R}^i \colon \bigotimes^2 V \to \bigotimes^2 V$ are defined on generators by
\begin{align*}
    g(\mathcal{R}^1(X \otimes Y), Z \otimes W) & = R(X,Y,Z,W), \\
    g(\mathcal{R}^2(X \otimes Y), Z \otimes W) & = R(X,Z,W,Y)
\end{align*}
and extended bilinearly. Note that $\mathcal{R}^i$ are self-adjoint operators with respect to $g.$
\end{definition}

\begin{remark}
\label{RelationshipToCurvatureOperators}
    \normalfont
From the definition, we easily compute that
\begin{align*}
    g( \mathcal{R}^1( X \wedge Y), Z \wedge W) & = 4 R(X,Y,Z,W), \\
    g( \mathcal{R}^1( X \wedge Y), Z \odot W) & = 0, \\
    g( \mathcal{R}^1( X \odot Y), Z \odot W) & = 0, \\
    g( \mathcal{R}^2( X \wedge Y), Z \wedge W) & = - \frac{1}{2} g( \mathcal{R}^1( X \wedge Y), Z \wedge W), \\
    g( \mathcal{R}^2( X \wedge Y), Z \odot W) & = 0, \\
    g( \mathcal{R}^2( X \odot Y), Z \odot W) & = 2 \left( R(X,Z,W,Y) + R(X,W,Z,Y) \right).
\end{align*}
In particular, we have the induced operators
\begin{align*}
    \mathcal{R}^1 {}_{| \Ext^2 V} & \colon \Ext^2 V \to \Ext^2 V, \\
    \mathcal{R}^2 {}_{| \bigodot^2 V} & \colon  \bigodot^2 V \to \bigodot^2 V.
\end{align*}
Note that $\mathcal{R}^1{}_{| \Ext^2 V} = 2 \ \mathfrak{R}$ is (two times) the curvature operator and $\mathcal{R}^2{}_{| \bigodot^2 V} =  \overline{R}$ as defined in \cite{NPWBettiNumbersAndCOSK} so that $ \operatorname{pr}_{S^2_0(V)} \circ \mathcal{R}^2 {}_{|S_0^2(V)} = \mathcal{R}$ is the curvature operator of the second kind. 

Moreover, for a complex vector space $(V,J,g)$ we have that 
\begin{align*}
\mathcal{R}^2{}_{| \bigodot^2 V^{1,0}} \colon \bigodot^2 V^{1,0} \to \bigodot^2 V^{1,0}
\end{align*}
and this operator agrees with the Calabi curvature operator since
\begin{align*}
    g( \mathcal{R}^2( X \odot Y), \overline{Z} \odot \overline{W}) = & \ 2 \left( R(X, \overline{Z}, \overline{W},Y) + R(X, \overline{W}, \overline{Z},Y) \right) \\
    = & \ 4 R(X, \overline{Z}, \overline{W}, Y) \\
    = & \ g( \mathcal{C}( X \odot Y), \overline{Z} \odot \overline{W})
\end{align*}
for $X, Y, Z, W \in V^{1,0}.$
\end{remark}

\begin{definition}
    Let $T \in \bigotimes^k V^{*}$ be a tensor. For a subspace $\mathfrak{g} \subset \mathfrak{gl}(V \otimes_{\mathbf{R}} \mathbf{C})$ we define $T^{\mathfrak{g}} \in \bigotimes^k V^{*} \otimes \mathfrak{g}^*$ by 
\begin{align*}
    T^{\mathfrak{g}}(X_1, \ldots, X_k)(L) = (LT)(X_1, \ldots, X_k)
\end{align*}
for all $L\in \mathfrak{g}$.

\end{definition}
   Explicitly, if $\Xi_{\alpha}$ is a basis for $\mathfrak{g},$ then 
    \begin{align*}
       T^{\mathfrak{g}} = \sum_{\alpha} \Xi_{\alpha}T \otimes \Xi_{\alpha}^*
    \end{align*}
    where $\Xi_{\alpha}$ acts on $T$ as a derivation and $\Xi_{\alpha}^{*}$ denotes the dual map. Note that if $E\in \End(\mathfrak{g})$ is an endomorphism, then $E(T^{\mathfrak{g}}) = \sum_{\alpha} \Xi_{\alpha}T \otimes E^*(\Xi_{\alpha}^*).$ If $\{\Xi_{\alpha}\}$ forms a unitary basis, i.e., if $g( \Xi_{\alpha}, \overline{ \Xi_{\beta}}) = \delta_{\alpha \beta},$ then, under the induced identification of $ \mathfrak{g}^*$ with $ \overline{\mathfrak{g}},$ we have $T^{\mathfrak{g}} = \sum_{\alpha} \Xi_{\alpha}T \otimes \overline{\Xi_{\alpha}}.$ Moreover, if $R$ is self-adjoint with respect to $g,$ then $R(T^{\mathfrak{g}}) = \sum_{\alpha} \Xi_{\alpha}T \otimes R(\overline{\Xi_{\alpha}}).$ 
    
    In the following we will consider in particular $\varphi^{\mathfrak{gl}},$ $\varphi^{\mathfrak{gl}_{\mathbf{C}}}$, where $\mathfrak{gl}_{\mathbf{C}} = \mathfrak{gl} \otimes_{\mathbf{R}} \mathbf{C},$ and for a complex vector space $(V,J,g)$ also $\varphi^{\bigodot^{2} V^{1,0}}.$ 

\begin{proposition}
\label{CurvatureR2gl}
    With respect to the tensor norm on $\bigotimes^pV$ we have for $\varphi \in \Ext^p V$
    \begin{align*}
        g( \mathcal{R}^2( \varphi^{\mathfrak{gl}}), \varphi^{\mathfrak{gl}}) = - \frac{p(p-1)}{2} \sum_{i,j,k,l} \sum_{i_3, \ldots, i_p} R_{ijkl} \ \varphi_{ij i_3 \ldots i_p} \varphi_{k l i_3 \ldots i_p}.
    \end{align*}
\end{proposition}
\begin{proof}
    Note that $(e_i \otimes e_j) e_k = \delta_{ik} e_j$ and hence
    \begin{align*}
        g( (e_i \otimes e_j) \varphi, (e_k \otimes e_l) \varphi) = & \ \sum_{r,s} \sum_{i_1, \ldots, i_p} \delta_{i i_r} \delta_{k i_s} \varphi_{i_1 \ldots j \ldots i_p} \varphi_{i_1 \ldots l \ldots i_p} \\
        = & \ \sum_{r \neq s} \sum_{i_1, \ldots, i_p} \delta_{i i_r} \delta_{k i_s} \varphi_{j i_s i_1 \ldots \widehat{j} \ldots \widehat{i_s} \ldots i_p} \varphi_{i_r l i_1 \ldots \widehat{i_r} \ldots \widehat{l} \ldots i_p} \\
        & \ + \sum_r \sum_{i_1, \dots, i_p} \delta_{i i_r} \delta_{k i_r} \varphi_{j i_1 \ldots \widehat{j} \ldots i_p} \varphi_{l i_1 \ldots \widehat{l} \ldots i_p} \\
        = & \ p(p-1) \sum_{i_3, \ldots, i_{p}}  \varphi_{j k i_3 \ldots i_{p}} \varphi_{i l i_3 \ldots i_{p}} + p \sum_{i_2, \dots, i_{p}} \delta_{i k} \varphi_{j i_2 \ldots i_{p}} \varphi_{l i_2 \ldots i_{p}}.
    \end{align*}
    Since 
    \begin{align*}
        \varphi^{\mathfrak{gl}} = \sum_{i,j} \left( (e_i \otimes e_j)  \varphi \right) \otimes (e_i \otimes e_j)
    \end{align*}
    and $g(\mathcal{R}^2(e_i \otimes e_j),e_k \otimes e_l) = R_{iklj}$ it follows that
    \begin{align*}
        g( \mathcal{R}^2( \varphi^{\mathfrak{gl}}), \varphi^{\mathfrak{gl}}) = & \ \sum_{i,j,k,l} g(\mathcal{R}^2(e_i \otimes e_j),e_k \otimes e_l) g( (e_i \otimes e_j) \varphi, (e_k \otimes e_l) \varphi) \\
        = & \ p(p-1) \sum_{i,j,k,l} \sum_{i_3, \ldots, i_{p}} R_{iklj} \ \varphi_{j k i_3 \ldots i_{p}} \varphi_{i l i_3 \ldots i_{p}} \\
        & \ + p \sum_{i,j,k,l} \sum_{i_2, \dots, i_{p}} \delta_{i k} R_{iklj} \ \varphi_{j i_2 \ldots i_{p}} \varphi_{l i_2 \ldots i_{p}} \\
        = & \ p(p-1) \sum_{i,j,k,l} \sum_{i_3, \ldots, i_{p}} R_{iklj} \ \varphi_{j k i_3 \ldots i_{p}} \varphi_{i l i_3 \ldots i_{p}} \\
        = & \ p(p-1) \sum_{i,j,k,l} \sum_{i_3, \ldots, i_{p}} ( - R_{klij} - R_{likj} ) \ \varphi_{j k i_3 \ldots i_{p}} \varphi_{i l i_3 \ldots i_{p}} \\
        = & \ - p(p-1) \sum_{i,j,k,l} \sum_{i_3, \ldots, i_{p}} R_{jlik} \ \varphi_{kj i_3  \ldots i_{p}} \varphi_{i l i_3 \ldots i_{p}} \\
        & \ - p(p-1) \sum_{i,j,k,l} \sum_{i_3, \ldots, i_{p}} R_{iljk} \ \varphi_{j k i_3 \ldots i_{p}} \varphi_{i l i_3 \ldots i_{p}} \\
        = & \ - p(p-1) \sum_{i,j,k,l} \sum_{i_3, \ldots, i_{p}} R_{iklj} \ \varphi_{j k i_3 \ldots i_{p}} \varphi_{i l i_3 \ldots i_{p}} \\
        & \ - p(p-1) \sum_{i,j,k,l} \sum_{i_3, \ldots, i_{p}} R_{ijkl} \ \varphi_{ij i_3 \ldots i_{p}} \varphi_{k l i_3 \ldots i_{p}} \\
        = & \ - g( \mathcal{R}^2( \varphi^{\mathfrak{gl}}), \varphi^{\mathfrak{gl}})  - p(p-1) \sum_{i,j,k,l} \sum_{i_3, \ldots, i_{p}} R_{ijkl} \ \varphi_{ij i_3 \ldots i_{p}} \varphi_{k l i_3 \ldots i_{p}}.
    \end{align*}
    Comparing the first and last lines completes the proof.
\end{proof}

\begin{proposition}
\label{RicLAndR2}
    If $\varphi \in \Ext^pV,$ then 
    \begin{align*}
        \frac{3}{2} g( \Ric_L( \varphi), \varphi) = g( \mathcal{R}^2( \varphi^{S^2}), \varphi^{S^2}) + p \sum_{i,j} \sum_{i_2, \ldots, i_{p}} R_{ij} \varphi_{i i_2 \ldots i_{p}} \varphi_{j i_2 \ldots i_{p}}.
    \end{align*}
\end{proposition}
\begin{proof}
Combining the Bochner formula on forms in  
\cite{BesseEinstein,NPWBettiNumbersAndCOSK} with Proposition~\ref{CurvatureR2gl} we have
\begin{align*}
    g( \Ric_L(\varphi), \varphi) = & \ - \frac{p(p-1)}{2}  \sum_{i,j,k,l} \sum_{i_3, \ldots, i_{p}} R_{ijkl} \ \varphi_{ij i_3 \ldots i_{p}} \varphi_{k l i_3 \ldots i_{p}} + p \sum_{i,j} \sum_{i_2, \ldots, i_{p}} R_{ij} \varphi_{i i_2 \ldots i_p} \varphi_{j i_2 \ldots i_p} \\
    = & \ g( \mathcal{R}^2( \varphi^{\mathfrak{gl}}), \varphi^{\mathfrak{gl}}) + p \sum_{i,j} \sum_{i_2, \ldots, i_{p}} R_{ij} \varphi_{i i_2 \ldots i_p} \varphi_{j i_2 \ldots i_p}.
\end{align*}

Furthermore, Remark~\ref{RelationshipToCurvatureOperators} implies  
\begin{align*}
    g( \mathcal{R}^2(\varphi^{\mathfrak{gl}}), \varphi^{\mathfrak{gl}}) & = g( \mathcal{R}^2(\varphi^{S^2}), \varphi^{S^2}) + g( \mathcal{R}^2(\varphi^{\mathfrak{so}}), \varphi^{\mathfrak{so}}) \\
    & = g( \mathcal{R}^2(\varphi^{S^2}), \varphi^{S^2}) - \frac{1}{2} g( \mathcal{R}^1(\varphi^{\mathfrak{so}}), \varphi^{\mathfrak{so}}) \\
    & = g( \mathcal{R}^2(\varphi^{S^2}), \varphi^{S^2}) - \frac{1}{2} g( \Ric_L(\varphi), \varphi)
\end{align*}
and the claim follows.
\end{proof}

\begin{remark}
\label{TranslationRicLCurvOperator}
\normalfont 
    Note that indeed
\begin{align*}
    g( \mathcal{R}^1(\varphi^{\mathfrak{so}}), \varphi^{\mathfrak{so}}) = 2 g( \mathfrak{R} (\varphi^{\mathfrak{so}}), \varphi^{\mathfrak{so}}) = g( \Ric_L(\varphi), \varphi)
\end{align*}
due to \cite[Lemma 9.3.3]{PetersenRiemannianGeometry} and the definition of the curvature operator in Section~\ref{SectionCurvatureTensors}. 
\end{remark}

\begin{lemma}
\label{CurvatureTermWithCalabi}
    Let $R$ be an algebraic K\"ahler curvature operator on $(V,J,g)$ with $\dim_{\mathbf{R}}V=2n$. Let $\Sigma_{\nu} \in \bigodot^2 V^{1,0}$ be an eigenbasis for the associated Calabi curvature operator $\mathcal{C}$ with corresponding eigenvalues $\sigma_{\nu}$ and let $\varphi \in \Ext^p (V^{\mathbf{C}})^{*}$ be a real form. If the basis $\{ \Sigma_{\nu} \}$ is unitary, i.e., if $g(\Sigma_{\nu}, \overline{\Sigma_{\mu}})=\delta_{\nu \mu},$ then
    \begin{align*}
        g(\Ric_L(\varphi), \varphi) = 2 \sum_{\nu=1}^{\binom{n+1}{2}} \sigma_{\nu} | \Sigma_{\nu} \varphi |^2.
    \end{align*}
\end{lemma}
\begin{proof}
    Recall from the proof of Proposition~\ref{RicLAndR2} that
\begin{align*}
g( \Ric_L(\varphi), \varphi) - p \sum_{i,j} \sum_{|I|=p-1} R_{ij} \varphi_{i i_2 \ldots i_p} \varphi_{ji_2 \ldots i_p} = g \left( \mathcal{R}^2 ( \varphi^{\mathfrak{gl}}), \varphi^{\mathfrak{gl}} \right) = 
    g \left( \mathcal{R}^2 ( \varphi^{\mathfrak{gl}_{\mathbf{C}}}), \overline{\varphi^{\mathfrak{gl}_{\mathbf{C}}}} \right),
\end{align*}
where $\mathfrak{gl}_{\mathbf{C}} = \mathfrak{gl} \otimes_{\mathbf{R}} {\mathbf{C}}.$

To compute $g \left( \mathcal{R}^2 ( \varphi^{\mathfrak{gl}_{\mathbf{C}}}), \overline{\varphi^{\mathfrak{gl}_{\mathbf{C}}}} \right)$ we first note that due to Remark~\ref{TraceInZCoordinates}
\begin{align*}
    g( (Z_a & \otimes \overline{Z_b}) \varphi, (\overline{Z_c} \otimes Z_d) \varphi) = \\
    = & \ \sum_{r, s} \sum_{i_1, \ldots, i_p} \varphi (e_{i_1}, \ldots, (Z_a \otimes \overline{Z_b}) e_{i_r}, \ldots, e_{i_p}) \varphi (e_{i_1}, \ldots, ( \overline{Z_c} \otimes Z_d ) e_{i_s}, \ldots, e_{i_p}) \\
    = & \ \sum_{r \neq s} \sum_{e,f} \sum_{|I|=p-2} \left\lbrace \varphi(e_{i_1}, \ldots, (Z_a \otimes \overline{Z_b}) Z_e, \ldots, Z_f, \ldots, e_{i_p}) \varphi (e_{i_1}, \ldots, \overline{Z_e}, \ldots, ( \overline{Z_c} \otimes Z_d ) \overline{Z_f}, \ldots, e_{i_p}) \right. \\
    & \hspace{2.8cm} + \varphi(e_{i_1}, \ldots, (Z_a \otimes \overline{Z_b}) Z_e, \ldots, \overline{Z_f}, \ldots, e_{i_p}) \varphi (e_{i_1}, \ldots, \overline{Z_e}, \ldots, ( \overline{Z_c} \otimes Z_d ) Z_f, \ldots, e_{i_p}) \\
    & \hspace{2.8cm} + \varphi(e_{i_1}, \ldots, (Z_a \otimes \overline{Z_b}) \overline{Z_e}, \ldots, Z_f, \ldots, e_{i_p}) \varphi (e_{i_1}, \ldots, Z_e, \ldots, ( \overline{Z_c} \otimes Z_d ) \overline{Z_f}, \ldots, e_{i_p}) \\
    & \hspace{2.8cm}  \left. + \varphi(e_{i_1}, \ldots, (Z_a \otimes \overline{Z_b}) \overline{Z_e}, \ldots, \overline{Z_f}, \ldots, e_{i_p}) \varphi (e_{i_1}, \ldots, Z_e, \ldots, ( \overline{Z_c} \otimes Z_d ) Z_f, \ldots, e_{i_p}) \right\rbrace \\
    & \ + \sum_{r} \sum_{e} \sum_{|I|=p-1} \left\lbrace \varphi (e_{i_1}, \ldots, (Z_a \otimes \overline{Z_b}) Z_e, \ldots, e_{i_p}) \varphi (e_{i_1}, \ldots, ( \overline{Z_c} \otimes Z_d ) \overline{Z_e}, \ldots, e_{i_p}) \right. \\
    & \hspace{3.2cm} \left. + \varphi (e_{i_1}, \ldots, (Z_a \otimes \overline{Z_b}) \overline{Z_e}, \ldots, e_{i_p}) \varphi (e_{i_1}, \ldots, ( \overline{Z_c} \otimes Z_d ) Z_e, \ldots, e_{i_p}) \right\rbrace \\
    = & \ \sum_{r \neq s} \sum_{e,f} \sum_{|I|=p-2}  \delta_{ae} \delta_{cf} \varphi_{\bar{b} \bar{f}i_3 \ldots i_p} \varphi_{e d i_3 \ldots i_p} + \sum_r \sum_{e} \sum_{|I|=p-1}  \delta_{a e} \delta_{c e} \varphi_{\bar{b} i_2 \ldots i_p} \varphi_{d i_2 \ldots i_p} \\
    = & \ p(p-1) \sum_{|I|=p-2} \varphi_{\bar{b} \bar{c} i_3 \ldots i_p} \varphi_{a d i_3 \ldots i_p} + p \sum_{|I|=p-1} \delta_{a c} \varphi_{\bar{b} i_2 \ldots i_p} \varphi_{d i_2 \ldots i_p}.
\end{align*}

Moreover, note that 
\begin{align*}
    g( \mathcal{R}^2( \overline{Z_a} \otimes Z_b), Z_c \otimes \overline{Z_d} ) = R_{\bar{a} c \bar{d} b}
\end{align*}
and 
\begin{align*}
    \sum_{b,c} R_{\bar{a} c \bar{d} b} \varphi_{\bar{b} \bar{c} i_3 \ldots i_p} \varphi_{a d i_3 \ldots i_p} = 0
\end{align*}
since $R_{\bar{a} c \bar{d} b} = R_{\bar{a} b \bar{d} c}$ but $\varphi_{\bar{b} \bar{c} i_3 \ldots i_p} = - \varphi_{\bar{c} \bar{b} i_3 \ldots i_p}.$
It follows that 
\begin{align*}
    \sum_{a,b,c,d}  g( & ( Z_a \otimes \overline{Z_b} )\varphi, (\overline{Z_c} \otimes Z_d ) \varphi )  \ g( \mathcal{R}^2( \overline{Z_a} \otimes Z_b), Z_c \otimes \overline{Z_d}) = \\ 
    = & \ p(p-1) \sum_{a,b,c,d} \sum_{|I|=p-2} R_{\bar{a} c \bar{d} b} \varphi_{\bar{b} \bar{c} i_3 \ldots i_p} \varphi_{a d i_3 \ldots i_p} + p \sum_{a,b,c,d} \sum_{|I|=p-1} \delta_{a c} R_{\bar{a} c \bar{d} b} \varphi_{\bar{b} i_2 \ldots i_p} \varphi_{d i_2 \ldots i_p} \\
    = & \ p \sum_{a,b,d} \sum_{|I|=p-1} R_{\bar{a} a \bar d b} \varphi_{\bar{b} i_2 \ldots i_p} \varphi_{d i_2 \ldots i_p} \\
    = & \ - p \sum_{b,d} \sum_{|I|=p-1} R_{b \bar{d}} \varphi_{\bar{b} i_2 \ldots i_p} \varphi_{d i_2 \ldots i_p} \\
    = & \ - \frac{p}{2} \sum_{i,j} \sum_{|I|=p-1} R_{ij} \varphi_{i i_2 \ldots i_p} \varphi_{ji_2 \ldots i_p}.
\end{align*}
The decomposition $V_{\mathbf{C}} = V^{1,0} \oplus V^{0,1}$ induces a decomposition of $\varphi^{\mathfrak{gl}_{\mathbf{C}}},$ namely
\begin{align*}
    \varphi^{\mathfrak{gl}_{\mathbf{C}}} = & \ \sum_{a,b} (Z_a \otimes Z_b) \varphi \otimes  (\overline{Z_a} \otimes \overline{Z_b}) + 
    \sum_{a,b} (Z_a \otimes \overline{Z_b}) \varphi \otimes  (\overline{Z_a} \otimes Z_b) \\ 
    & \ + \sum_{a,b} (\overline{Z_a} \otimes Z_b) \varphi \otimes  (Z_a \otimes \overline{Z_b}) +
    \sum_{a,b} ( \overline{Z_a} \otimes \overline{Z_b}) \varphi \otimes  (Z_a \otimes Z_b) \\
    = & \ \varphi^{V^{1,0} \otimes V^{1,0}} + \varphi^{V^{1,0} \otimes V^{0,1}} +\varphi^{V^{0,1} \otimes V^{1,0}} +
    \varphi^{V^{0,1} \otimes V^{0,1}}.
\end{align*}

Note that $\mathcal{R}^2$ operates on $\varphi^{\mathfrak{gl}_{\mathbf{C}}}$ by acting on the component in $V_{\mathbf{C}} \otimes V_{\mathbf{C}}.$ Since $R$ is K\"ahler, we have, for example
\begin{align*}
    g( \mathcal{R}^2( Z_a \otimes Z_b ), Z_c \otimes \overline{Z_d}) & = R(Z_a, Z_c, \overline{Z_d}, Z_b ) = 0 \ \text{ or } \\
    g( \mathcal{R}^2( Z_a \otimes \overline{Z_b} ), Z_c \otimes \overline{Z_d}) & = R(Z_a, Z_c, \overline{Z_d}, \overline{Z_b} ) = 0.
\end{align*}
It follows that 
\begin{align*}
    g \left( \mathcal{R}^2 ( \varphi^{\mathfrak{gl}_{\mathbf{C}}}), \overline{\varphi^{\mathfrak{gl}_{\mathbf{C}}}} \right) =  
    & \ g \left( \mathcal{R}^2 ( \varphi^{V^{1,0} \otimes V^{1,0}}),  \overline{\varphi^{V^{1,0} \otimes V^{1,0}}} \right) + g \left( \mathcal{R}^2 ( \varphi^{V^{0,1} \otimes V^{0,1}}),  \overline{\varphi^{V^{0,1} \otimes V^{0,1}}} \right) \\
    & + \ g \left( \mathcal{R}^2 ( \varphi^{V^{1,0} \otimes V^{0,1}} ),  \overline{\varphi^{V^{1,0} \otimes V^{0,1}}} \right) + g \left( \mathcal{R}^2 ( \varphi^{V^{0,1} \otimes V^{1,0}} ),  \overline{\varphi^{V^{0,1} \otimes V^{1,0}}} \right) \\
    = & \ 2 g \left( \mathcal{R}^2 ( \varphi^{V^{1,0} \otimes V^{1,0}}),  \overline{\varphi^{V^{1,0} \otimes V^{1,0}}} \right) + 2 g \left( \mathcal{R}^2 ( \varphi^{V^{1,0} \otimes V^{0,1}} ),  \overline{\varphi^{V^{1,0} \otimes V^{0,1}}} \right).
    \end{align*}
    
Furthermore, 
\begin{align*}
    \varphi^{V^{1,0} \otimes V^{1,0}} = \frac{1}{4} \sum_{a,b} (Z_a \odot Z_b) \varphi \otimes (\overline{Z_a} \odot \overline{Z_b}) + \frac{1}{2} \sum_{a < b} (Z_a \wedge Z_b) \varphi \otimes (\overline{Z_a} \wedge \overline{Z_b}) = \varphi^{{\bigodot^{2} V^{1,0}}} + \varphi^{\Ext^2 V^{1,0}}.
\end{align*}
From Remark~\ref{RelationshipToCurvatureOperators} we obtain
\begin{align*}
    g \left( \mathcal{R}^2 ( \varphi^{V^{1,0} \otimes V^{1,0}} ),  \overline{\varphi^{V^{1,0} \otimes V^{1,0}}} \right) = & \ g \left( \mathcal{R}^2 ( \varphi^{\bigodot^{2} V^{1,0}}), \overline{\varphi^{\bigodot^{2} V^{1,0}}} \right) + g \left( \mathcal{R}^2 ( \varphi^{\Ext^2 V^{1,0}}), \overline{\varphi^{\Ext^2 V^{1,0}}} \right) \\
    = & \ g \left( \mathcal{R}^2 ( \varphi^{\bigodot^{2} V^{1,0}}), \overline{\varphi^{\bigodot^{2} V^{1,0}}} \right)  - \frac{1}{2} g \left( \mathcal{R}^1 ( \varphi^{\Ext^2 V^{1,0}}), \overline{\varphi^{\Ext^2 V^{1,0}}} \right) \\
    = & \ g \left( \mathcal{R}^2 ( \varphi^{\bigodot^{2} V^{1,0}}), \overline{\varphi^{\bigodot^{2} V^{1,0}}} \right)
\end{align*}
since any K\"ahler curvature operator vanishes on $\Ext^2 V^{1,0}.$

Overall we obtain
\begin{align*}
    g \left( \mathcal{R}^2 ( \varphi^{\mathfrak{gl}_{\mathbf{C}}}), \overline{\varphi^{\mathfrak{gl}_{\mathbf{C}}}} \right) = & \ 2 \sum_{a,b,c,d}  g( ( Z_a \otimes \overline{Z_b} )\varphi, (\overline{Z_c} \otimes Z_d ) \varphi )  \ g( \mathcal{R}^2( \overline{Z_a} \otimes Z_b), Z_c \otimes \overline{Z_d}) \\ 
    & \ + 2 g \left( \mathcal{R}^2 ( \varphi^{\bigodot^{2} V^{1,0}}), \overline{\varphi^{\bigodot^{2} V^{1,0}}} \right) \\
    = & \ - p \sum_{i,j} \sum_{|I|=p-1} R_{ij} \varphi_{i i_2 \ldots i_p} \varphi_{ji_2 \ldots i_p} + 2 g \left( \mathcal{R}^2 ( \varphi^{\bigodot^{2} V^{1,0}}), \overline{\varphi^{\bigodot^{2} V^{1,0}}} \right)
\end{align*}
and hence, from Remark~\ref{RelationshipToCurvatureOperators}, we have
\begin{align*}
    g( \Ric_{L}( \varphi), \varphi ) = & \ 2 g \left( \mathcal{R}^2 ( \varphi^{\bigodot^{2} V^{1,0}}), \overline{\varphi^{\bigodot^{2} V^{1,0}}} \right) 
    =  2 g \left( \mathcal{C} ( \varphi^{\bigodot^{2} V^{1,0}}), \overline{\varphi^{\bigodot^{2} V^{1,0}}} \right).
\end{align*}
\end{proof}

\section{Main estimates and proofs of Theorems~\ref{thm:main-theorem-1}--\ref{thm:main-theorem-3}}
\label{SectionProofMainTheorems}

In this section, we provide the proofs of the main theorems. The starting point is the expression for the curvature term in the Bochner-Weitzenb\"ock formula on forms,
   \begin{align*}
        g(\Ric_L(\varphi), \varphi) = 2 \sum_{\nu} \sigma_{\nu} | \Sigma_{\nu} \varphi |^2
    \end{align*}
where $\Sigma_a$ is a local unitary frame  with $\mathcal{C}(\Sigma_a) = \sigma_a \Sigma_a$, as established in Lemma~\ref{CurvatureTermWithCalabi}. Following the general principles of \cite{NPWBettiNumbersAndCOSK,PW1,PW2}, the strategy is to compute $| \varphi^{\bigodot^{2} V^{1,0}} |^2  =  \sum_a  | \Sigma_a \varphi |^2$ for a primitive real form $\varphi$ and to establish an estimate of the form $| S \varphi |^2 \leq \frac{1}{\Upsilon} | S |^2 | \varphi^{\bigodot^{2} V^{1,0}} |^2.$ This is accomplished in Lemma~\ref{NormHolomSymHatForms} and Proposition~\ref{MainEstimateForForms}. The weight principle in Proposition~\ref{WeightPrinciple} then provides conditions on eigenvalues for the Calabi curvature operator so that the curvature term $g(\Ric_L(\varphi), \varphi)$ is nonnegative.

\begin{proposition}
\label{NormZBarZinserted}
    If $\varphi \in \Ext^{p,q} V^{*},$ then
    \begin{align*}
        (p+q)(p+q-1) \sum_{a,b} \left| \iota_{Z_a} \iota_{\overline{Z_b}} \varphi \right|^2 = pq | \varphi |^2.
    \end{align*}
\end{proposition}
\begin{proof}
    Note that if $Z^I \wedge \overline{Z^J} = Z^{i_1} \wedge \ldots \wedge Z^{i_p} \wedge \overline{Z^{j_1}} \wedge \ldots \wedge \overline{Z^{j_q}},$ then
    \begin{align*}
    \left| Z^I \wedge \overline{Z^J} \right|^2 & = (p+q)!, \\
        \left| \iota_{Z_a} \iota_{\overline{Z_b}} Z^I \wedge \overline{Z^J} \right|^2 & = \begin{cases}
            (p+q-2)! & \ \text{ if } a \in I \text{ and } b \in J, \\
            0 & \ \text{ otherwise}
        \end{cases}
    \end{align*}
    and the claim follows.
\end{proof}

\begin{lemma}
\label{NormHolomSymHatForms}
    A real form $\varphi \in \Ext^{p,q} V^{*} \oplus \Ext^{q,p} V^{*}$ satisfies
    \begin{align*}
        \left| \varphi^{\bigodot^{2} V^{1,0}} \right|^2  =   \frac{1}{4}\left((p+q)(n+1) - 2pq\right)|\varphi |^2 - \frac{1}{2(p+q)(p+q-1)}|\Lambda( \varphi) |^2,
    \end{align*}
where $\Lambda$ is the formal adjoint of the Lefschetz operator. In particular, if $\varphi$ is primitive, then
\begin{align*}
\left| \varphi^{\bigodot^{2} V^{1,0}} \right|^2 = \frac{1}{4} \left((p+q)(n+1) - 2pq\right)|\varphi|^2. 
\end{align*}
\end{lemma}
\begin{proof}
    Let $k=p+q$ and note that due to Remark~\ref{TraceInZCoordinates} we have 
    \begin{align*}
        g( (Z_a & \odot Z_b) \varphi, (\overline{Z_a} \odot \overline{Z_b}) \varphi) = \\
    = & \ \sum_{r, s} \sum_{i_1, \ldots, i_k} \varphi (e_{i_1}, \ldots, (Z_a \odot Z_b) e_{i_r}, \ldots, e_{i_k}) \varphi (e_{i_1}, \ldots, ( \overline{Z_a} \odot \overline{Z_b} ) e_{i_s}, \ldots, e_{i_k}) \\
    = & \ \sum_{r \neq s} \sum_{c,d} \sum_{|I|=k-2} \left\lbrace \varphi(e_{i_1}, \ldots, (Z_a \odot Z_b) Z_c, \ldots, Z_d, \ldots, e_{i_k}) \varphi (e_{i_1}, \ldots, \overline{Z_c}, \ldots, ( \overline{Z_a} \odot \overline{Z_b} ) \overline{Z_d}, \ldots, e_{i_k}) \right. \\
    & \hspace{2.8cm} + \varphi(e_{i_1}, \ldots, (Z_a \odot Z_b) Z_c, \ldots, \overline{Z_d}, \ldots, e_{i_k}) \varphi (e_{i_1}, \ldots, \overline{Z_c}, \ldots, ( \overline{Z_a} \odot \overline{Z_b} ) Z_d, \ldots, e_{i_k}) \\
    & \hspace{2.8cm} + \varphi(e_{i_1}, \ldots, (Z_a \odot Z_b) \overline{Z_c}, \ldots, Z_d, \ldots, e_{i_k}) \varphi (e_{i_1}, \ldots, Z_c, \ldots, ( \overline{Z_a} \odot \overline{Z_b} ) \overline{Z_d}, \ldots, e_{i_k}) \\
    & \hspace{2.8cm}  \left. + \varphi(e_{i_1}, \ldots, (Z_a \odot Z_b) \overline{Z_c}, \ldots, \overline{Z_d}, \ldots, e_{i_k}) \varphi (e_{i_1}, \ldots, Z_c, \ldots, ( \overline{Z_a} \odot \overline{Z_b} ) Z_d, \ldots, e_{i_k}) \right\rbrace \\
    & \ + \sum_{r} \sum_{c} \sum_{|I|=k-1} \left\lbrace \varphi (e_{i_1}, \ldots, (Z_a \odot Z_b) Z_c, \ldots, e_{i_k}) \varphi (e_{i_1}, \ldots, ( \overline{Z_a} \odot \overline{Z_b} ) \overline{Z_c}, \ldots, e_{i_k}) \right. \\
    & \hspace{3.2cm} \left. + \varphi (e_{i_1}, \ldots, (Z_a \odot Z_b) \overline{Z_c}, \ldots, e_{i_k}) \varphi (e_{i_1}, \ldots, ( \overline{Z_a} \odot \overline{Z_b} ) Z_c, \ldots, e_{i_k}) \right\rbrace \\
    = & \ \sum_{r \neq s} \sum_{c,d} \sum_{|I|=k-2}  \left( \delta_{ac} \varphi_{b \bar{d} i_3 \ldots i_k} + \delta_{bc} \varphi_{a \bar{d} i_3 \ldots i_k} \right) \left( \delta_{ad} \varphi_{c \bar{b} i_3 \ldots i_k} + \delta_{bd} \varphi_{c \bar{a} i_3 \ldots i_k} \right) \\
    & \ + \sum_r \sum_{c} \sum_{|I|=k-1} \left( \delta_{ac} \varphi_{b i_2 \ldots i_k} + \delta_{bc} \varphi_{a i_2 \ldots i_k} \right) \left( \delta_{ac} \varphi_{\bar{b} i_2 \ldots i_k} + \delta_{bc} \varphi_{\bar{a} i_2 \ldots i_k} \right)  \\
    = & \ \sum_{r \neq s} \sum_{|I|=k-2} 2 \left( \varphi_{a \bar{b} i_3 \ldots i_k} \varphi_{b \bar{a} i_3 \ldots i_k} +  \varphi_{a \bar{a} i_3 \ldots i_k} \varphi_{b \bar{b} i_3 \ldots i_k} \right) \\
    & \ + \sum_{r} \sum_{|I|=k-1} \left( \varphi_{a i_2 \ldots i_k} \varphi_{\bar{a} i_2 \ldots i_k}  + \varphi_{b i_2 \ldots i_k} \varphi_{\bar{b} i_2 \ldots i_k} + 2 \delta_{ab} \varphi_{a i_2 \ldots i_k} \varphi_{\bar{a} i_2 \ldots i_k} \right) \\
     = & \ k(k-1) \sum_{|I|=k-2} 2 \left( \varphi_{a \bar{b} i_3 \ldots i_k} \varphi_{b \bar{a} i_3 \ldots i_k} +  \varphi_{a \bar{a} i_3 \ldots i_k} \varphi_{b \bar{b} i_3 \ldots i_k} \right) \\
    & \ + k \sum_{|I|=k-1} \left( \varphi_{a i_2 \ldots i_k} \varphi_{\bar{a} i_2 \ldots i_k}  + \varphi_{b i_2 \ldots i_k} \varphi_{\bar{b} i_2 \ldots i_k} + 2 \delta_{ab} \varphi_{a i_2 \ldots i_k} \varphi_{\bar{a} i_2 \ldots i_k} \right).
    \end{align*}
Thus,
\begin{align*}
    4 \left| \varphi^{\bigodot^{2} V^{1,0}} \right|^2 = & \  \sum_{a,b} g( (Z_a \odot Z_b) \varphi, (\overline{Z_a} \odot \overline{Z_b}) \varphi) \\
    = & \ - 2 k(k-1) \sum_{a,b} \sum_{|I|=k-2} \left( \varphi_{a \bar{b} i_3 \ldots i_k} \varphi_{\bar{a} b i_3 \ldots i_k} +  \varphi_{a \bar{a} i_3 \ldots i_k} \varphi_{\bar{b} b i_3 \ldots i_k} \right) \\
    & \ + 2k(n+1) \sum_a \sum_{|I|=k-1} \varphi_{a i_2 \ldots i_k} \varphi_{\bar{a} i_2 \ldots i_k}.
\end{align*}
Note that the formal adjoint of the Lefschetz map satisfies 
\begin{align*}
    \left( \Lambda \varphi \right) (v_1, \ldots, v_{k-2}) = & \ k(k-1) \sum_{a} \varphi(e_a, e_{a+n}, v_1, \ldots, v_{k-2}) \\
    = & \ - \sqrt{-1} k(k-1) \sum_a \varphi(Z_a, \overline{Z_a}, v_1, \ldots, v_{k-2}).
\end{align*}
Together with Proposition~\ref{NormZBarZinserted} and Remark~\ref{TraceInZCoordinates}  this implies
\begin{align*}
    4 \left| \varphi^{\bigodot^{2} V^{1,0}} \right|^2 = & \  -2 pq | \varphi|^2 - \frac{2}{k(k-1)} | \Lambda \varphi|^2 + k(n+1) | \varphi |^2 \\
    = & \ (k(n+1)-2pq) | \varphi |^2 - \frac{2}{k(k-1)} | \Lambda \varphi|^2.
\end{align*}
\end{proof}

\begin{proposition}
\label{NormalFormHolomSym}
Let $(V,J,g)$ be a Euclidean vector space with a compatible complex structure and $S \in \bigodot^2 V^{1,0}$. There exists a unitary basis $Z_1, \ldots, Z_n$ for $V^{1,0}$ and constants $\rho_1, \ldots, \rho_n \geq 0$ such that 
\begin{align*}
    S = \sum_{a=1}^n \rho_a Z_a \otimes Z_a.
\end{align*}
\end{proposition}
\begin{proof}
Note that via $X \odot Y \mapsto g( \text{}  \cdot \text{},  \overline{X}) Y + g( \text{}  \cdot \text{}, \overline{Y})X$ we can assign to $S \in \bigodot^2 V^{1,0}$ a complex linear operator $\hat{S} : V^{1,0} \to V^{1,0}$. By construction, $\hat{S}$ is symmetric with respect to the Hermitian inner product $g( \cdot, \overline{\text{} \cdot \text{}})$ on $V^{1,0}$. It follows that there is a unitary basis $Z_1, \ldots, Z_n$ for $V^{1,0}$ and $\rho_1, \ldots, \rho_n \in \mathbf{R}$ such that $\hat{S}(Z_a) = \rho Z_a.$ It follows that $S = \sum_a \rho_a Z_a \otimes Z_a$ and by replacing $Z_a$ with $\sqrt{-1}Z_a$ we may assume $\rho_a \geq 0.$
\end{proof}

\begin{proposition}
\label{ActionOnPQforms}
The action of $\bigotimes^2 V^{1,0}$ on $\Ext^{p,q} V^{*}$ is given by
\begin{align*}
    (Z_a \otimes Z_b) & (Z^{i_1} \otimes \ldots \otimes Z^{i_p} \otimes \overline{Z}^{j_1} \otimes \ldots \otimes \overline{Z}^{j_q})  = \\
    = & \ - \sum_{k=1}^p \delta_{b i_k} Z^{i_1} \otimes \ldots \otimes \overline{Z^a} \otimes \ldots \otimes Z^{i_p} \otimes \overline{Z}^{j_1} \otimes \ldots \otimes \overline{Z}^{j_q}.
\end{align*}
\end{proposition}
\begin{proof}
Note that
    \begin{align*}
            (Z_a \otimes Z_b) Z^c & = - Z^c( (Z_a \otimes Z_b) \cdot ) = - g(Z_a, \cdot ) Z^c(Z_b) = - \delta_{bc} \overline{Z^a}, \\
    (Z_a \otimes Z_b) \overline{Z^c} & = - \overline{Z^c}( (Z_a \otimes Z_b) \cdot ) = - g(Z_a, \cdot ) \overline{Z^c}(Z_b) = 0
    \end{align*}
    follow from the definition of the action on tensors in \eqref{DefinitionActionOnTensors}.
\end{proof}

\begin{remark}
\label{NotationForForms}
    \normalfont
    For two multi-indices $I = (i_1, \ldots,  i_p)$ and $J = (j_1, \ldots, j_q)$ we define $K$ to be the multi-index with $p+q$ entries, consisting of the indices of $I$ and $\overline{J}.$ The entries of $K$ are ordered lexicographically and if $i_k = j_l,$ then $K = ( \ldots, i_k, \bar{j}_l, \ldots ).$ 

    Conversely, given a multi-index $K$ as above, we recover $I_K$ and $J_K$ so that $I_K$ and $\overline{J_K}$ yield $K.$

    Furthermore, given $K$ we denote by $Z^K$ the corresponding $(p,q)$-form with $| Z^K |^2 = 1$. Note that $\varphi= Z^1 \wedge \ldots \wedge Z^p \wedge \overline{Z^1} \wedge \ldots \wedge \overline{Z^q}$ satisfies $| \varphi |^2 = (p+q)!.$ By convention, if $K = \emptyset$, we set $Z^K = 0$. 
    
    For example, if $I=(2,3)$ and $J=(1,2),$ then $K=(\bar{1}, 2, \bar{2}, 3)$ and $Z^K= \frac{1}{\sqrt{24}} \ \overline{Z^1} \wedge Z^2 \wedge \overline{Z^2} \wedge Z^3.$ 

    Since $Z^K=0$ if any index appears twice in $I$ or $J$, we may assume throughout that this is not the case.
\end{remark}

\begin{proposition}
\label{MainEstimateForForms}
If $\psi \in \Ext^{p,q} V^{*} \oplus \Ext^{q,p} V^{*}$ is real, then
\[
|S\psi|^2 \leq \left(\frac{1}{2} + \min\left( p,q,\frac{\sqrt{pq}}{2} \right)\right)|S|^2|\psi|^2
\]
for all $S \in \bigodot^2 V^{1,0}$.  In particular, if $\psi$ is primitive, then
\[
|S\psi|^2 \leq \frac{2 + 4\min\left(p,q,\frac{\sqrt{pq}}{2} \right)}{(p+q)(n+1) - 2pq}|S|^2 |\psi^{\bigodot^{2} V^{1,0}}|^2
\]
for all $S \in \bigodot^2 V^{1,0}$.
\end{proposition}
\begin{proof}
    The result is a combination of three estimates. To start, due to Proposition~\ref{NormalFormHolomSym}, we may assume that $S = \sum_a \lambda_a Z_a \otimes Z_a$ with $\lambda_a \geq 0.$ With the notation introduced in Remark~\ref{NotationForForms}, consider a $(p,q)$-form $\varphi = \sum_K \varphi_K Z^K.$ Proposition~\ref{ActionOnPQforms} shows that
\begin{align*}
    S \varphi = - \sum_a \sum_K \lambda_a \varphi_K Z^{K_{a \mapsto \bar{a}}},
\end{align*}
where $K_{a \mapsto \bar{a}}$ is constructed from $K$ by replacing $a$ with $\bar{a}$ if $a \in I_K$ and $K_{a \mapsto \bar{a}}=\emptyset$ otherwise. Note that this turns a $(p,q)$-form into a $(p-1,q+1)$-form.

Cauchy-Schwarz, together with the fact that there are at most $p$ possibilities for $a,$ yields \begin{align}
\label{estimate1}
\left| S\varphi \right|^2 \leq  \left( \sum_a \lambda_a^2 \right)   \left(\sum_a\left| \sum_{K} \varphi_{K} Z^{K_{a \mapsto \bar{a}}} \right|^2 \right) \leq | S |^2 \ p \sum_K | \varphi_K |^2 =  p | S |^2 | \varphi |^2.
\end{align}
Note that $Z^{K}=0$ if an index appears in $I_K$ or $J_K$ twice.

For the second estimate, we relabel indices giving
\begin{align*}
S \varphi &= - \sum_a \sum_K \lambda_a \varphi_K Z^{K_{a \mapsto \bar{a}}}  = - \sum_{L} \left( \sum_{b \in J_L \setminus I_L} \lambda_b \varphi_{L_{\bar{b} \mapsto b}} \right) Z^{L} = - \sum_{L} \left( \sum_{b} \lambda_b \varphi_{L_{\bar{b} \mapsto b}} \right) Z^{L} \\
&= - \sum_{L} \sum_{a} \lambda_a \varphi_{L_{\bar{a} \mapsto a}} Z^{L},
\end{align*}
where $I_L$, $J_L$ are the index sets so that $I_L$ and $\overline{J_L}$ yield $L$ as in Remark~\ref{NotationForForms}.

Consequently, we get 
 \begin{align*}
 \left| S\varphi \right|^2 \ &= \sum_{L} \left| \sum_{b} \lambda_b \varphi_{L_{\bar{b} \mapsto b}} \right|^2
 =\sum_L \sum_{a,b}\lambda_a\lambda_b\varphi_{L_{\bar{a} \mapsto a}}\overline{\varphi_{L_{\bar{b} \mapsto b}}}\\
 &=\sum_K \sum_{a,b}\lambda_a\lambda_b\varphi_{K}\overline{\varphi_{K_{a \mapsto \bar{a},\bar{b} \mapsto b}}}
 =\sum_{a,b}\lambda_a\lambda_b \sum_K \varphi_{K}\overline{\varphi_{K_{a \mapsto \bar{a},\bar{b} \mapsto b}}},
 \end{align*}
where $K_{a \mapsto \bar{a},\bar{b} \mapsto b}$ denotes $K$  if $a=b\in I_K\setminus J_K$  and otherwise consists of the simultaneous replacements $a \mapsto \bar{a}$ and $\bar{b} \mapsto b$.
For two indices $a,b$, define 
\begin{align*}
\varphi_{ab} = \sum\limits_{K:a\in I, b\in J}\varphi_{K}Z^{K} \ \text{ and } \ \omega_{ab} = \sum\limits_{K:a\in I, b\in J}\varphi_{K_{a \mapsto \bar{a},\bar{b} \mapsto b}}Z^{K}
\end{align*}
and note that

\begin{align*}
    \varphi_{ba} &= \sum\limits_{K:b\in I, a\in J}\varphi_{K}Z^K
    =\sum\limits_{K:a\in I, b\in J}\varphi_{K_{a \mapsto \bar{a},\bar{b} \mapsto b}}Z^{K_{a \mapsto \bar{a},\bar{b} \mapsto b}}
\end{align*}
implies $|\omega_{ab}|=|\varphi_{ba}|$ and hence
\begin{align*}
    \sum_K \varphi_{K}\overline{\varphi_{K_{a \mapsto \bar{a},\bar{b} \mapsto b}}} &= g( \varphi_{ab}, \overline{\omega_{ab}}) \le |\varphi_{ab}||\omega_{ab}| = |\varphi_{ab}||\varphi_{ba}|.
\end{align*}

Using the inequality $\tr(A^2)\leq ||A||_2^2=\sum_{i,j} a_{ij}^2$ on the matrix $(\lambda_a |\varphi_{ab}|)_{ab}$, we obtain

\begin{align}
    \left| S\varphi \right|^2 \ &= \sum_{a,b}\lambda_a\lambda_b \sum_K \varphi_{K}\overline{\varphi_{K_{a \mapsto \bar{a},\bar{b} \mapsto b}}} \nonumber\\ 
    &\le \sum_{a,b}\lambda_a \lambda_b|\varphi_{ab}||\varphi_{ba}| \nonumber\\
    &\le \sum_{a,b}\lambda_a^2|\varphi_{ab}|^2 \nonumber\\
    &\le \sum_{a}\lambda_a^2 (q+1)|\varphi|^2 \nonumber\\
    &= (q+1)|S|^2|\varphi|^2, \label{estimate2}
\end{align}
noting that for a fixed $a$, a given $\varphi_K$ appears in no more than $(q+1)$ many $\varphi_{ab}$.

Applying the estimates \eqref{estimate1} and \eqref{estimate2} to a real form $\psi = \varphi + \overline{\varphi} \in \Ext^{p,q}V^* \oplus \Ext^{q,p}V^*$, we obtain
\begin{align*}
 \left| S  \psi \right|^2 = \left|S \varphi\right|^2 + \left|S \overline{\varphi} \right|^2 \leq  \left( \frac{1}{2} + \min(p,q) \right) | S |^2 | \psi |^2.
\end{align*}

For the third estimate, note that the real form $\psi$ satisfies
\begin{align}
\left| S\psi \right|^2 = & \ \sum_{L} \left| \sum_{b \in J_L \setminus I_L} \lambda_b \psi_{L_{\bar{b} \mapsto b}} \right|^2\ = \  \sum_{L} \sum_{a,b \in J_L \setminus I_L} \lambda_a \lambda_b \ \psi_{L_{\bar{a} \mapsto a}} \overline{\psi_{L_{\bar{b} \mapsto b}}} \nonumber\\ 
=& \sum_{K} \sum_{a\in I_K\setminus J_K}\sum_{b \in (J_K \setminus I_K)\cup \{a\}} \lambda_a \lambda_b \ \psi_K \overline{\psi_{K_{a \mapsto \bar{a}, \bar{b} \mapsto b}}}\nonumber \\
= & \ \sum_K \sum_{a \in I_K \setminus J_K} \lambda_a^2 | \psi_K |^2 + \sum_{K} \sum_{\substack{a \in I_K \setminus J_K \\ b \in J_K \setminus I_K}} \lambda_a \lambda_b \ \psi_K \overline{\psi_{K_{a \mapsto \bar{a}, \bar{b} \mapsto b}}}\nonumber
\\= & \ \sum_{K\in (p,q)} \left(\sum_{a \in I_K \setminus J_K} \lambda_a^2 | \psi_K |^2 +  \sum_{\substack{a \in I_K \setminus J_K \\ b \in J_K \setminus I_K}} \lambda_a \lambda_b \ \psi_K \overline{\psi_{K_{a \mapsto \bar{a}, \bar{b} \mapsto b}}}\right) \nonumber \\
& + \ \sum_{K\in(q,p)} \left( \sum_{a \in I_K \setminus J_K} \lambda_a^2 | \psi_K |^2 + \sum_{\substack{a \in I_K \setminus J_K \\ b \in J_K \setminus I_K}} \lambda_a \lambda_b \ \psi_K \overline{\psi_{K_{a \mapsto \bar{a}, \bar{b} \mapsto b}}}\right)\nonumber\\
= & \ \sum_{K\in (p,q)} \Bigg( \sum_{a \in I_K \setminus J_K} \lambda_a^2 | \psi_K |^2 +  \sum_{\substack{a \in I_K \setminus J_K \\ b \in J_K \setminus I_K}} \lambda_a \lambda_b \ \psi_K \overline{\psi_{K_{a \mapsto \bar{a}, \bar{b} \mapsto b}}} \nonumber\\
& \hspace{18mm} + \sum_{a \in J_K \setminus I_K} \lambda_a^2 | \psi_{\overline{K}} |^2 + \sum_{\substack{a \in J_K \setminus I_K \\ b \in I_K \setminus J_K}} \lambda_a \lambda_b \ \psi_{\overline{K}} \overline{\psi_{\overline{K}_{a \mapsto \bar{a}, \bar{b} \mapsto b}}}\Bigg)
\nonumber\\
= & \ \sum_{K\in (p,q)} \Bigg( \sum_{a \in I_K \setminus J_K \cup J_K \setminus I_J} \lambda_a^2 | \psi_K |^2 +  \sum_{\substack{a \in I_K \setminus J_K \\ b \in J_K \setminus I_K}} 2 \lambda_a \lambda_b \  \operatorname{Re}(\psi_K \overline{\psi_{K_{a \mapsto \bar{a}, \bar{b} \mapsto b}}})\Bigg), \label{longformula}
\end{align}
where $K \in (p,q)$ means $K$ is an index set corresponding to a $(p,q)$-form. 
We wish to maximize $| S\psi |^2$ subject to the constraints $| S |^2 =1$ and $| \psi |^2= 1$.

The constraint $|\psi|^2 = 1$ implies there exists $\alpha \in \R$ such that, at the critical point $\psi$,
 \begin{align*}
      2\alpha \psi = \operatorname{grad}\left(\phi \mapsto g(S\phi,\overline{ S \phi})\right) = 2 \overline{S} (S \psi). 
 \end{align*}
Hence, $\alpha = g( \overline{S}(S\psi),\overline{\psi}) = |S\psi|^2$. Since we are at a critical point, this lets us read off the gradient from \eqref{longformula}, and we get for each $K$ that

\begin{align*}
2 \alpha \psi_K =  \sum_{a \in I_K \setminus J_K \cup J_K \setminus I_J} \lambda_a^2 \psi_K  +  \sum_{\substack{a \in I_K \setminus J_K \\ b \in J_K \setminus I_K}} 2 \lambda_a \lambda_b \ \psi_{K_{a \mapsto \bar{a}, \bar{b} \mapsto b}}.
\end{align*}

Looking at the $K$ such that $|\psi_K|$ is maximal, we get
\begin{align*}
2\alpha |\psi_K| \leq \left( \sum_{a \in I_K \setminus J_K \cup J_K \setminus I_K} \lambda_a^2 +  \sum_{\substack{a \in I_K \setminus J_K \\ b \in J_K \setminus I_K}} 2 \lambda_a \lambda_b \ \right) |\psi_K|
\end{align*}
and hence
\begin{align*}
    2\alpha \leq \sum_{a \in I_K \setminus J_K \cup J_K \setminus I_K} \lambda_a^2 +  \sum_{\substack{a \in I_K \setminus J_K \\ b \in J_K \setminus I_K}} 2 \lambda_a \lambda_b.
\end{align*}

The right hand side is maximized when 
\begin{align*}
    \lambda_i = \begin{cases}
        \frac{1}{\sqrt{2| I \setminus J |}} & \text{ if } i\in I\setminus J, \\
        \frac{1}{\sqrt{2| J \setminus I |}} & \text{ if } i\in J\setminus I, \\
        0 & \text{ otherwise.}
    \end{cases}
\end{align*}

It follows that 
\begin{align*}
\alpha  \leq   \frac{1}{2} + \frac{\sqrt{| I \setminus J |\cdot| J \setminus I |}}{2} \le \frac{1}{2} + \frac{\sqrt{pq}}{2},
\end{align*}
which yields the claim when combined with the previous two estimates.
\end{proof}

\begin{remark}
    \normalfont
    For given $(p,q)$, the form $\operatorname{Re}(\sum_K Z^K)$, where $K$ runs over all $(p,q)$-index sets such that $I_K \cup J_K = \{1, \ldots, p+q\}$, satisfies $| S \psi |^2 = \left( \frac{1}{2} + \frac{pq}{p+q} \right) | S |^2 | \psi |^2$ for $S=\sum_{a=1}^{p+q} Z_a \otimes Z_a$, which we suspect to be the actual optimal constant. On $(p,0)$, $(0,q)$, and $(p,p)$-forms, this agrees with the estimate in Proposition~\ref{MainEstimateForForms}, which is therefore sharp in these degrees.
\end{remark}

\begin{proposition}\label{WeightPrinciple}
Suppose that $\mathcal{C} : \bigodot^2 V^{1,0} \to \bigodot^2 V^{1,0}$ is an algebraic Calabi curvature operator and let $\kappa \leq 0.$ If the eigenvalues $\sigma_1 \leq \ldots \leq \sigma_{\binom{n+1}{2}}$ of $\mathcal{C}$ satisfy
\begin{align*}
    \label{eqn:lower-bound}
\sigma_1 + \ldots + \sigma_{\lfloor \Upsilon \rfloor} + \left( \Upsilon - \lfloor \Upsilon \rfloor  \right) \sigma_{\lfloor \Upsilon\rfloor +1}  \geq  \kappa \Upsilon,
\end{align*}
then for any primitive real form $\varphi \in \Ext^{p,q}(V^{\ast}) \oplus \Ext^{q,p}(V^{\ast})$, we have
\begin{align*}
 g(\Ric_L(\varphi), \varphi) = 2 \sum_{a} \sigma_a | \Sigma_a \varphi |^2 \geq 2 \kappa | \varphi^{\bigodot^{2} V^{1,0}} |^2.
\end{align*}
If $\mathcal{C}$ is $\Upsilon$-positive, then $g(\Ric_L(\varphi), \varphi) \geq 0$ and equality holds if and only if $\varphi =0$.
\end{proposition}
\begin{proof}
The first part is a consequence of Lemma~\ref{CurvatureTermWithCalabi} and the proof of \cite[Lemma 2.1]{PW1}, see also \cite[Lemma 3.4]{NPWBettiNumbersAndCOSK}.
Note that Proposition~\ref{MainEstimateForForms} asserts that $| S \varphi |^2 \leq \frac{1}{\Upsilon} | S |^2 | \varphi^{\bigodot^{2} V^{1,0}} |^2$ for all $S \in \bigodot^2 V^{1,0}$. We note that 
\begin{align*}
\sum_{a=1}^{\binom{n+1}{2}} \sigma_a | \Sigma_a \varphi |^2 
 \geq & \ \sigma_{\lfloor \Upsilon \rfloor +1} \sum_{a = \lfloor \Upsilon \rfloor +1}^{\binom{n+1}{2}} | \Sigma_a \varphi |^2 + \sum_{a=1}^{\lfloor \Upsilon \rfloor } \sigma_a | \Sigma_a \varphi |^2 \\
= & \ \sigma_{\lfloor \Upsilon \rfloor +1}  | \varphi^{\bigodot^{2} V^{1,0}} |^2 + \sum_{a=1}^{\lfloor \Upsilon \rfloor} (\sigma_a - \sigma_{\lfloor \Upsilon \rfloor +1}) | \Sigma_a \varphi |^2 \\
\geq & \ \sigma_{\lfloor \Upsilon \rfloor +1}  | \varphi^{\bigodot^{2} V^{1,0}} |^2 + \frac{1}{\Upsilon} \sum_{a=1}^{\lfloor \Upsilon \rfloor} (\sigma_a - \sigma_{\lfloor \Upsilon \rfloor +1} )  | \varphi^{\bigodot^{2} V^{1,0}} |^2 \\
= & \ \frac{1}{\Upsilon} \left( (\Upsilon - \lfloor \Upsilon \rfloor) \sigma_{\lfloor \Upsilon \rfloor +1} + \sigma_{\lfloor \Upsilon \rfloor} + \ldots + \sigma_1  \right)  | \varphi^{\bigodot^{2} V^{1,0}} |^2 \\
\geq & \ \kappa | \varphi^{\bigodot^{2} V^{1,0}} |^2.
\end{align*} 
For the final claim, note that nonnegativity follows by setting $\kappa=0.$ If furthermore $\mathcal{C}$ is $\Upsilon$-positive, then the above computation shows 
\begin{align*}
    g(\Ric_L(\varphi), \varphi) \geq \frac{2}{\Upsilon} \left( \sigma_1 + \ldots + \sigma_{\lfloor \Upsilon \rfloor} + \left( \Upsilon - \lfloor \Upsilon \rfloor  \right) \sigma_{\lfloor \Upsilon\rfloor +1} \right) c(p,q) | \varphi|^2 >0
\end{align*}
unless $\varphi = 0,$ since $| \varphi^{\bigodot^{2} V^{1,0}} |^2 = c(p,q) | \varphi|^2$ for primitive forms according to Lemma~\ref{NormHolomSymHatForms}.
\end{proof}

Before proving our main theorems, we begin with a local holonomy classification for $\pm \frac n2$- nonnegative Calabi curvature operators.

\begin{thm}\label{thm:main-theorem-2.2}
If $X^n$ is a (not necessarily complete) K\"ahler manifold with $\mathcal{C}$ or $-\mathcal{C}$ being  $\frac{n}{2}$-nonnegative, then one of the following holds.

\begin{itemize}
\item[(i)] The holonomy is irreducible and is either $\msf{U}(n)$ or the space is locally  isometric to the complex quadric $\msf{SO}(n+2)/ \left( \msf{SO}(2) \times \msf{SO}(n) \right)$ with its symmetric metric, or its dual $\mathsf{SO}_0(n,2) /( \mathsf{SO}(2)\times \mathsf{SO}(n))$.
\item[(ii)] The holonomy is reducible and given by  $\msf{U}(n_1)\times \ldots \times \msf{U}(n_k)$ where $n_1+\ldots+n_k\leq n$.
\end{itemize}
\end{thm}

\begin{proof}
Starting with case (i), suppose that the holonomy is irreducible. If $X$ is not locally symmetric, then from Berger's classification, the restricted holonomy must be $\mathsf{U}(n), \mathsf{SU}(n)$ or $\mathsf{Sp}(n)$. The latter two imply $X$ is Ricci-flat, but since $\pm \cop$ is $\frac n2$-nonnegative, this in turn implies that $X^n$ is flat, contradicting irreducibility, so $X$ must have holonomy $\mathsf{U}(n)$. If $X$ is locally symmetric, then the eigenvalues of $\cop$ are known by \cite{CV,BCV}. In particular, the only Hermitian symmetric spaces with $\frac{n}{2}$-nonnegative $\pm\cop$ are the complex quadric $\mathsf{SO}(n+2)/ \left( \mathsf{SO}(2) \times \mathsf{SO}(n) \right)$ and its dual $\mathsf{SO}_0(n,2) / (\mathsf{SO}(2)\times \mathsf{SO}(n))$, which completes the proof of part (i).

For part (ii), suppose that the holonomy representation is reducible. Then some neighborhood around each $x\in X$ is biholomorphically isometric to a product $X_0^{n_0} \times X_1^{n_1} \times \ldots \times X_k^{n_k}$ of K\"ahler manifolds where $X_0$ is flat and each $X_i$, $1 \leq i \leq k$, has irreducible holonomy representation. It is easily checked that the Calabi curvature operator vanishes on the subbundle $\bigoplus_{i < j} T^{1,0}X_i \odot T^{1,0}X_j$, hence $\dim \ker \cop \geq d := n_0 + \sum_{i < j}n_i n_j$. It follows that the operator $\pm \cop_1 \oplus \ldots \oplus \cop_k \in \operatorname{End}\left(\bigoplus_{i} T^{1,0}X_i \odot T^{1,0}X_i\right)$ is $\left( \frac n2 - d \right)$-nonnegative, where each $\cop_i$ is the Calabi curvature operator of $X_i$. However since $n = n_0 + \sum_in_i$, it holds that $\frac{n}{2} - d \leq 0$, so $\pm \cop_1 \oplus \ldots \oplus \cop_k$, and hence each $\pm \cop_i, 1 \leq i \leq k$, is actually nonnegative. Repeating the argument of part (i) for each factor implies $X_i$ has holonomy $\mathsf{U}(n_i)$ or is locally symmetric. However, by \cite{CV,BCV}, the only irreducible Hermitian symmetric spaces with $\pm \cop$ nonnegative are complex projective space or complex hyperbolic space, which have holonomy $\mathsf{U}(n)$. Thus, the holonomy is given by $\mathsf{U}(n_1)\times \ldots \times \mathsf{U}(n_k)$.
\end{proof}

We now turn to proving our main theorems for the Calabi curvature operator. \vspace{2mm}

\textbf{Proof of Theorem~\ref{thm:main-theorem-3}.} Suppose $X$ is a compact K\"ahler manifold with $\Upsilon_{p,q}$-nonnegative Calabi curvature operator for some pair $(p,q)$. If $\varphi \in \Omega^{p,q}_X \oplus \Omega^{q,p}_X$ is a real primitive harmonic form, then the Bochner formula \eqref{prelim:Bochner-formula} and Proposition~\ref{WeightPrinciple} imply
\[
g(\nabla^*\nabla \varphi,\varphi) = -g(\operatorname{Ric}_L(\varphi),\varphi)\leq 0,
\]
hence $0 \geq \int_X g(\nabla^*\nabla \varphi,\varphi) = \int_X |\nabla \varphi|^2$, so $\varphi$ is parallel. Moreover, $g(\operatorname{Ric}_L (\varphi) ,\varphi) = 0$. If the Calabi curvature operator is in addition $\Upsilon_{p,q}$-positive, then Proposition~\ref{WeightPrinciple} implies $\varphi$ must vanish identically. \hfill \qed \vspace{2mm}

\textbf{Proof of Theorem~\ref{thm:main-theorem-1}.} By Serre duality and the Hodge decomposition theorem for K\"ahler manifolds, it suffices to show that all real primitive harmonic $(p,q)$-forms vanish whenever $1 \leq p + q \leq n$. This will follow from Theorem~\ref{thm:main-theorem-3} once we show that $\Upsilon_{p,q} \geq \frac{n}{2}$ for all $(p,q)$ such that $1 \leq p + q \leq n$. This is equivalent to showing
\[
n\left(p+q - 2 \min\left(p,q,\frac{\sqrt{pq}}{2}\right) - 1\right) + p+q - 2pq \geq 0,
\]
which is elementary. \hfill \qed \vspace{2mm}

\textbf{Proof of Theorem~\ref{thm:main-theorem-2}.}
First, by Theorem~\ref{thm:main-theorem-3} and the fact that $\Upsilon_{p,q} \geq \frac{n}{2}$ for all $1 \leq p+q \leq n$, all real primitive harmonic forms on $X$ are parallel. Starting with part (i), suppose that the holonomy of $X$ is irreducible. By Theorem~\ref{thm:main-theorem-2.2}, either the holonomy of $X$ is $\mathsf{U}(n)$, or $X$ is locally isometric to $\mathsf{SO}(n+2)/ \left( \mathsf{SO}(2)\times \mathsf{SO}(n+2) \right)$. If we are in the former situation, then all primitive parallel forms vanish. Hence all real primitive harmonic forms vanish and so $X$ must have the rational cohomology of $\mathbf{P}^n$. On the other hand, if $X$ is locally isometric to the complex quadric, then it has positive Ricci curvature, hence it is Fano and therefore simply connected, so the isometry is global.

For part (ii), suppose the holonomy of $X$ is reducible. By the proof of Theorem~\ref{thm:main-theorem-2.2}, we have that $\cop$ is nonnegative, and the result follows by Remark~\ref{rmk:NonNegativeCalabi}.\hfill \qed

\section{The K\"ahler curvature operator and proof of Theorem~\ref{thm:main-theorem-4}}
\label{SectionKaehlerEinstein}

In this section,  we will show that the results achieved in \cite{PW2} can be improved if the metric is K\"ahler--Einstein and prove Theorem~\ref{thm:main-theorem-4}. \vspace{2mm}

Let $(V,J,g)$ be a Euclidean vector space of real dimension $2n$ with a compatible complex structure. Suppose that 
\begin{align*}
    \mathfrak{K} \colon \Ext^{1,1} V \to \Ext^{1,1} V
\end{align*}
is an algebraic K\"ahler curvature operator and let $\lambda_1 \leq \ldots \leq \lambda_{n^2}$ denote its eigenvalues. 

If $\mathfrak{K}$ is in addition K\"ahler-Einstein, $\Ric = \lambda g,$ then 
\begin{align*}
    \lambda = \frac{\scal}{2n} = \frac{\tr( \mathfrak{K})}{n} = \frac{1}{n} \sum\nolimits_{\alpha} \lambda_{\alpha}
\end{align*}
and the bi-vector corresponding to the K\"ahler form, $\omega_K = \frac{\sqrt{-1}}{2} \sum_{a=1}^n Z_a \wedge \overline{Z_a},$ is an eigenvector of $\mathfrak{K}$ with eigenvalue $\lambda.$ In particular, $\mathfrak{K}$ induces an operator 
\begin{align*}
    \mathfrak{K}_{|\mathfrak{su}(n)} \colon \mathfrak{su}(n) \to \mathfrak{su}(n).
\end{align*}
It will be useful to note that, since $\frac{1}{2} \scal = \tr(\mathfrak{K}) = \lambda + \tr(\mathfrak{K}_{| \mathfrak{su}(n)}) = \frac{1}{2n} \scal + \tr(\mathfrak{K}_{| \mathfrak{su}(n)})$, we have 
\begin{align*}
    \tr \left( \mathfrak{K}_{|\mathfrak{su}(n)} \right) = (n-1)\lambda. 
\end{align*}

\begin{lemma}
\label{suHatPrimitiveForm}
For a primitive $(p,q)$-form $\varphi \in \Ext^{p,q}V$,  we have 
\begin{align*}
\left| \varphi^{\mathfrak{su}} \right|^2 = \left( 2pq + (p+q)(n+1-(p+q)) - \frac{(p-q)^2}{n} \right) | \varphi |^2.
\end{align*}
\end{lemma} 
\begin{proof}
We recall from \cite[Proposition 3.2]{PW2} that a primitive form $\varphi \in \Ext^{p,q}V$ satisfies
\begin{align*}
\left| \varphi^{\mathfrak{u}} \right|^2 = \left( 2pq + (p+q)(n+1-(p+q) \right) | \varphi |^2.
\end{align*}
Furthermore, it follows from \cite[Proposition 1.3]{PW2} that the bi-vector corresponding to the K\"ahler form, $\omega_K = \frac{\sqrt{-1}}{2} \sum_{a=1}^n Z_a \wedge \overline{Z_a},$ acts on $(p,q)$-forms by multiplication with $\sqrt{-1} (p-q).$ Note that, moreover, $| \omega_K |^2 = n$ and hence
\begin{align*}
\left| (\omega_K) \varphi \right|^2  = \frac{(p-q)^2}{n} | \omega_K |^2 | \varphi |^2.
\end{align*} 
Thus, the decomposition  $\left| \varphi^{\mathfrak{u}} \right|^2 = \left| \frac{1}{\sqrt{n}} (\omega_K) \varphi \right|^2 + \left| \varphi^{\mathfrak{su}} \right|^2$ implies
\begin{align*}
\left| \varphi^{\mathfrak{su}} \right|^2 = \left( 2pq + (p+q)(n+1-(p+q) ) \right) | {\varphi} |^2 - \frac{(p-q)^2}{n} | {\varphi} |^2
\end{align*} 
for primitive $\varphi \in \Ext^{p,q}V$.
\end{proof}

\begin{remark}
\label{UnEstimateForms}
    \normalfont
    If $\varphi \in \Ext^{p,q}V$ is primitive and $L \in \mathfrak{u}(n)$, then \cite[Proposition 3.4]{PW2} shows that
\begin{align*}
    | L \varphi |^2 \leq (p+q) | L |^2 | {\varphi} |^2.
\end{align*}
\end{remark}

\begin{proposition}
\label{PropositionVanishingFormsKE}
    Suppose that $\mathfrak{K} \colon \Ext^{1,1}V \to \Ext^{1,1}V$ is an algebraic K\"ahler-Einstein curvature operator over $(V,J,g).$ If the induced curvature operator $\mathfrak{K}_{| \mathfrak{su}(n)} \colon \mathfrak{su}(n) \to \mathfrak{su}(n)$ is $\Gamma_{p,q}$-nonnegative, where
    \begin{align*}
        \Gamma_{p,q} = \frac{n(n^2-1)(p+q) - 2n(n-1)pq}{n(n-1)(p+q) + (p-q)^2},
    \end{align*}
    and $\varphi \in \Ext^{p,q}V$ is a primitive form, then
    \begin{align*}
        g( \Ric_L( \varphi), \overline{\varphi} ) \geq 0.
    \end{align*}    
\end{proposition}
\begin{proof}
    Since $\mathfrak{K}$ is K\"ahler-Einstein, we can find an orthonormal eigenbasis $\{ \Xi_a \}$ for $\mathfrak{K}_{|\mathfrak{su}(n)}$ with corresponding eigenvalues $\{ \mathfrak{\lambda}_{\alpha} \}.$ Hence, \cite[Proposition 1.6]{PW2} shows that
    \begin{align*}
       \frac{1}{2} g( \Ric_L( \varphi), \overline{\varphi} ) = \lambda \left| \frac{1}{\sqrt{n}} (\omega_K)\varphi \right|^2  + \sum_{\alpha} \lambda_{\alpha} | \Xi_{\alpha} \varphi |^2 ,
    \end{align*}
    where the additional factor of $\frac{1}{2}$ stems from the tensor norm convention as in Remark~\ref{TranslationRicLCurvOperator}. The previous computations thus imply
    \begin{align*}
        \frac{1}{2} g( \Ric_L( \varphi), \overline{\varphi} ) 
        = & \ \lambda \frac{(p-q)^2}{n} | {\varphi} |^2 + \sum_{\alpha} \lambda_{\alpha} | \Xi_{\alpha} \varphi |^2 \\
        = & \ \frac{(p-q)^2}{n(n-1)} \tr( \mathfrak{K}_{|\mathfrak{su}(n)} )  | {\varphi} |^2 + \sum_{\alpha} \lambda_{\alpha} | \Xi_{\alpha} \varphi |^2 \\
        = & \sum_{\alpha} \lambda_{\alpha} \left( \frac{(p-q)^2}{n(n-1)} | \varphi |^2 + | \Xi_{\alpha} \varphi |^2 \right).
    \end{align*}
Suppose that $| \varphi|^2 =1.$ In order to estimate the sum, we use the weight principle \cite[Theorem 3.6]{NPWBettiNumbersAndCOSK}. That is, we need to compare the total weight 
\begin{align*}
    \sum_{\alpha} \left( \frac{(p-q)^2}{n(n-1)} | \varphi |^2 + | \Xi_{\alpha} \varphi |^2 \right) = & \ \left(1 + \frac{1}{n} \right) (p-q)^2 | \varphi |^2 + \left| \varphi^{\mathfrak{su}(n)} \right|^2 \\ 
    = & \ (p-q)^2 + 2pq + (p+q)(n+1-(p+q))  \\
    = & \ (p+q)(n+1) - 2pq
\end{align*}
to the highest weight, which we can estimate by
\begin{align*}
    \frac{(p-q)^2}{n(n-1)} | \varphi |^2 + | \Xi_{\alpha} \varphi |^2 \leq \frac{(p-q)^2}{n(n-1)} + (p+q).
\end{align*}
Note that we used Lemma~\ref{suHatPrimitiveForm} and Remark~\ref{UnEstimateForms} in the computation. The weight principle  \cite[Theorem 3.6]{NPWBettiNumbersAndCOSK} now says that the weighted sum
\begin{align*}
    \sum_{\alpha} \lambda_{\alpha} \left( \frac{(p-q)^2}{n(n-1)} | \varphi |^2 + | \Xi_{\alpha} \varphi |^2 \right)
\end{align*}
is nonnegative provided that 
\begin{align*}
    \lambda_1 + \ldots + \lambda_{ \lfloor \Gamma_{p,q} \rfloor} + \left( \Gamma_{p,q} - \lfloor \Gamma_{p,q} \rfloor \right) \lambda_{\lfloor \Gamma_{p,q} \rfloor + 1} \geq 0,
\end{align*}
where $\Gamma_{p,q}$ is the quotient of the total weight by the highest weight.
\end{proof}

\begin{remark}
    \normalfont
    Note that $\Gamma_{n,0}=\frac{n^2-1}{n}$ and $\Gamma_{p,p}=n+1-p.$
\end{remark}

\textbf{Proof of Theorem~\ref{thm:main-theorem-4}.} Any  harmonic primitive form $\varphi \in \Omega_X^{p,q}$ satisfies the Bochner identity
\begin{align*}
    \Delta \frac{1}{2} | \varphi |^2 = | \nabla \varphi |^2 + g ( \Ric_L( \varphi), \overline{\varphi} ).
\end{align*}
If $\Gamma_{p,q} \geq 0,$ then Proposition~\ref{PropositionVanishingFormsKE} implies that $g ( \Ric_L( \varphi), \overline{\varphi} ) \geq 0$ and thus $\varphi$ is parallel. If $\Gamma_{p,q}>0,$ then $\varphi$ vanishes. To show that all primitive harmonic forms vanish, we may restrict to $p+q \leq n$ due to Serre duality. Theorem~\ref{thm:main-theorem-4} now follows from the observation that all $\Gamma_{p,q}$ are nonnegative (respectively positive), if $\mathfrak{K}_{| \mathfrak{su}(n)} \colon \Ext^{1,1}_0 TX \to \Ext^{1,1}_0 TX$ is $\left( \frac{n}{2} + 1 \right)$-nonnegative (respectively $\left( \frac{n}{2} + 1 \right)$-positive). \hfill $\Box$


\end{document}